\pdfoutput=1
\documentclass[12pt]{amsart}
\usepackage{epsfig,amssymb,pdfsync,color}

\newtheorem{Thm}[equation]{Theorem}
\newtheorem{Cor}[equation]{Corollary}
\newtheorem{Lem}[equation]{Lemma}
\newtheorem{Pro}[equation]{Proposition}

\theoremstyle{definition}

\newtheorem{Def}[equation]{Definition}

\theoremstyle{remark}

\newtheorem{Rem}[equation]{Remark}

\numberwithin{equation}{section}
\makeatletter
\renewcommand{\c@figure}{\c@equation}
\makeatother

\newcommand{\maths}[1]{{\bf #1}}

\newcommand{\T}{{\mathcal T}}

\newcommand{\RR}{\maths{R}}
\newcommand{\NN}{\maths{N}}

\renewcommand{\SS}{\maths{S}}

\newcommand{\ra}{\rightarrow}

\newcommand{\bord}{\partial}

\newcounter{fig}

\newcommand{\ack}{\noindent{\bf Acknowledgement.}}

\begin{document}

\title[Dirichlet Problem]{Boundary Value Problems on Planar Graphs and Flat Surfaces with integer cone singularities, I: The Dirichlet Problem}

\author{Sa'ar Hersonsky}

\address{Department of Mathematics\\ 
University of Georgia\\ 
Athens, GA 30602}

\urladdr{http://www.math.uga.edu/~saarh}
\email{saarh@math.uga.edu}
\thanks{}
\date{\today; version 0.0526100}

\begin{abstract}
Consider a planar, bounded, $m$-connected region $\Omega$, and let $\bord\Omega$ be its boundary.  Let $\mathcal{T}$ be a cellular decomposition of $\Omega\cup\bord\Omega$, where each $2$-cell is either a triangle or a quadrilateral. From these data and a conductance function we construct a canonical pair $(S_{},f)$ where $S$ is a genus $(m-1)$ {\it singular flat surface} tiled by rectangles and $f$ is an energy preserving mapping from ${\mathcal T}^{(1)}$ onto $S_{}$. 
By a singular flat surface, we will mean a surface which carries a metric structure locally modeled on the Euclidean plane, except at a finite number of points. These points have cone singularities, and the cone angle is allowed to take any positive value (see for instance \cite{Tro} for an excellent survey).  
Our realization may be considered as a discrete uniformization of planar bounded regions. 
\end{abstract}

\maketitle

\section{Introduction}
\label{se:Intro}

Consider a planar, bounded, $m$-connected region $\Omega$, and let $\bord\Omega$ be its boundary.  Let $\mathcal{T}$ be a cellular decomposition of $\Omega\cup\bord\Omega$, where each cell is either a triangle or a quadrilateral. Let $\bord\Omega=E_1\sqcup E_2$, where $E_1$ is the outermost component of $\bord\Omega$. Invoke a {\it conductance function} on ${\mathcal T^{(1)}}$, making it a {\it finite network}, and use it to define a combintorial Laplacian $\Delta$ on ${\mathcal T}^{(0)}$. Let $k$ be a positive constant. Let $g$ be the solution of a {\it Dirichlet boundary value problem} (D-BVP) defined on ${\mathcal T}^{(0)}$ and  determined by  requiring that $g|_{E_1}=k, g|_{E_2}=0$ and
$\Delta g=0$ at every interior vertex of ${\mathcal T}^{(0)}$, i.e. $g$ is {\it combinatorially  harmonic}. 
Furthermore, let $E(g)$ be the {\it Dirichlet energy} of $g$. Following the notation in \cite{ChGrYa}, let $\frac{\bord g}{\bord n}(x)$ denote the {\it normal derivative} of $g$ at the vertex $x\in\bord \Omega$.  Our first result is:

 \begin{Thm}  {\rm (Discrete uniformization of a pair of pants)} 
 \label{Th:pair} Let ${\mathcal P}$ be a bounded,  triply connected planar region and let  $(\Omega,\bord\Omega,{\mathcal T})= ({\mathcal P},\bord{\mathcal P}=E_1\sqcup E_2,{\mathcal T})$ where $E_2=E_2^1\sqcup E_2^2$.   Let $S_{\mathcal P}$ be a  singular flat pair of pants with exactly one singular point $Q$ having $4 \pi$ as its cone angle such that 
 \begin{enumerate}
 \item ${\mbox{\rm Length}_{\rm Euclidean}}({S_{\mathcal P}})_{E_1}= \ \ \ \sum_{x\in E_1}\frac{\bord g}{\bord n}(x)$, 
 \item ${\mbox{\rm Length}_{\rm Euclidean}}({S_{\mathcal P}})_{E_2^{1}}=- \sum_{x\in E_2^1}\frac{\bord g}{\bord n}(x)$, 
 \item ${\mbox{\rm Length}_{\rm Euclidean}}({S_{\mathcal P}})_{E_2^2}=-\sum_{x\in E_2^2}\frac{\bord g}{\bord n}(x)$, and
 \item the Euclidean length of a shortest geodesic connecting $E_1$ to either $E_2^1$ or $E_2^2$ is $k$, 
  \end{enumerate} 
 where $({S_{\mathcal P}})_{E_1}, ({S_{\mathcal P}})_{E_2^{1}}$ and $({S_{\mathcal P}})_{E_2^{2}}$ are the boundary components of $S_{\mathcal P}$.
  Then there exists
 a mapping $f$ which associates to each edge in  ${\mathcal T}^{(1)}$ a unique  Euclidean rectangle in $S_{\mathcal P}$,  in such a way that the collection of these rectangles forms a tiling of $S_{\mathcal P}$. 
 Furthermore, $f$ is energy preserving in the sense that $E(g)= {\rm Area}(S_{\mathcal P})$, and $f$ is boundary preserving.
 \end{Thm}

Throughout this paper a Euclidean rectangle will denote the image under an isometry of a planar Euclidean rectangle. For instance, most of the image rectangles that we will construct embed in a flat Euclidean cylinder. These will further be glued in a way that will not distort the Euclidean structure (see  
\S\ref{se:Annulus} and \S\ref{se:pop} for the details).
In our setting, boundary preserving means that the rectangle associated to an edge $[u,v]$ with $u\in \partial \Omega$ has one of its edges on a corresponding boundary component of the singular surface. 

A  singular flat, genus zero compact surface with $m\geq 3$ boundary components with conical singularities, will be called a ladder of singular pairs of pants. We now state the main result of this paper:

\begin{Thm} {\rm (The general case)}
  \label{Th:ladder}
 Let  $(\Omega,\bord\Omega,{\mathcal T})$ be a bounded, $m$-connected, planar region with $E_2=E_2^1\sqcup E_2^2\ldots \sqcup E_2^{m-1}$.  Let $S_{\Omega}$ be a ladder of singular pairs of pants such that
 \begin{enumerate}
 \item ${\mbox{\rm Length}_{\rm Euclidean}}(S_{\Omega})_{E_1}= \ \ \ \sum_{x\in E_1}\frac{\bord g}{\bord n}(x)$, and
  \item ${\mbox{\rm Length}_{\rm Euclidean}}(S_{\Omega})_{E_2^i}=- \sum_{x\in E_2^i}\frac{\bord g}{\bord n}(x)$, for $i=1,\cdots, m-1$,
 \end{enumerate}
 where $(S_{\Omega})_{E_1}$ and $(S_{\Omega})_{E_2^i}$, for $i=1,\cdots, m-1$, are the boundary components of $S_{\Omega}$.
   Then, there exists
   a mapping $f$ which associates to each edge in  ${\mathcal  T}^{(1)}$ a unique  Euclidean rectangle in $S_{\Omega}$ in such a way that the collection of these rectangles forms a tiling of $S_{\Omega}$. 
   Furthermore, $f$ is  boundary preserving, and $f$ is energy preserving in the sense that  $E(g)= {\rm Area}(S_{\Omega})$. 

 \end{Thm}

The following Corollary is straightforward, thus establishing the statement in the abstract of this paper 
 \begin{Cor} 
 \label{Cor:surface}
 Under the assumptions of Theorem~\ref{Th:ladder}, there exists a canonical pair $(S,f)$, where $S$ is a flat surface with conical singularities of genus $(m-1)$ tiled by Euclidean rectangles and $f$ is an energy preserving mapping from
 ${\mathcal T}^{(1)}$ into $S$, in the sense that $2E(g)={\rm Area}(S)$.   Moreover, $S$ admits a pair of pants decomposition whose dividing curves have Euclidean lengths given by ${\rm (1)-(3)}$ of Theorem~\ref{Th:ladder}.
 
 \end{Cor} 
 \begin{proof}
 
 Given $(\Omega,\bord\Omega,{\mathcal T})$, glue together two copies of $S_{\Omega}$ (their existence is guaranteed by Theorem~\ref{Th:ladder}) along corresponding boundary components.  This results in a flat surface $S=S_{\Omega}\bigcup\limits_{ \bord\Omega }S_{\Omega}$ of genus $(m-1)$ and a mapping ${\bar f}$ which restricts to $f$ on each copy.
 
\end{proof}
 
In the course of the proofs of Theorem~\ref{Th:pair} and Theorem~\ref{Th:ladder}, it will become apparent that the number of singular points and their cone angles, as well as the lengths of shortest geodesics between boundary curves in the ladder, may be explicitly determined. In particular, the cone angles obtained by our construction are always even integer multiples of $\pi$ (see Equation~(\ref{eq:cone}) for the actual computation). Some of these surfaces, those with trivial monodromy, are called {\it translation surfaces} and for excellent accounts see for instance  \cite{HuMaScZo},  \cite{Ma}  and \cite{Zo}.  Also, the dimensions of each rectangle are determined by the given D-BVP problem on 
${\mathcal  T}^{(0)}$.  Concretely, for $[u,v]\in {\mathcal  T}^{(1)}$, the associated rectangle will have its height equal to $(g(u)-g(v))$ and its width equal to $c(u,v)(g(u)-g(v))$, when $g(u)>g(v)$. Some of the rectangles constructed are not embedded.  We will comment on this point (which is also transparent in the proofs) in Remark~(\ref{re:embedd}). In a snapshot, some of the rectangles which arise from intersection of edges with singular  level curves are not  embedded.


\medskip
The following theorem is foundational and serves as a building block in the proofs of all of the above theorems. We prove:
\begin{Thm} {\rm (Discrete uniformization of an annulus \cite{BSST})} 
\label{Th:annulus} 
Let  ${\mathcal A}$ be an annulus and let $(\Omega,\bord\Omega,{\mathcal T})= ({\mathcal A},\bord{\mathcal A}=E_1\sqcup E_2,{\mathcal T})$.  Let $S_{\mathcal A}$ be a straight Euclidean cylinder with height $H=k$  and 
circumference $$C=\sum_{x\in E_1}\frac{\bord g}{\bord n}(x).$$
Then 
there exists a mapping $f$ which associates to each edge in  ${\mathcal T}^{(1)}$ a unique embedded Euclidean rectangle in $S_{\mathcal A}$ in such a way that the collection of these rectangles forms a tiling of  $S_{\mathcal A}$. Furthermore, $f$ is boundary preserving, and $f$ is energy preserving in the sense that $E(g)= {\rm Area}(S_{\mathcal A})$.
\end{Thm}
 
Given $(\Omega,\bord\Omega,{\mathcal T})$, we will work with the natural affine structure induced by the cellular decomposition. Let us denote this cell complex endowed with this affine structure by ${\mathcal C}(\Omega,\partial \Omega,{\mathcal T})$. 
 There is a common thread in our proofs of the theorems above. Extend $g$ to an affine map ${\bar g}$ defined on ${\mathcal T}$. This results in a piecewise linear structure on ${\mathcal T}$. Next, study the level curves of ${\bar g}$ on a $2$-dimensional complex which is homotopically equivalent  to ${\mathcal C}(\Omega,\partial \Omega,{\mathcal T})$, embedded in $\RR^{3}$, obtained by using ${\bar g}$ as a height function on ${\mathcal C}(\Omega,\bord\Omega,{\mathcal T})$. We will work with the level curves of  ${\bar g}$ or equivalently, with their projection on ${\mathcal C}(\Omega,\partial \Omega,{\mathcal T})$.
   The topological structure of the level curves associated with the solutions will be carefully studied.  For example, the level curves in the case of an annulus 
(Theorem~\ref{Th:annulus}) are simple, piecewise linear closed curves, all of which are  in the (free) homotopy  class determined by $E_1$. 
 One nice consequence of this is that  all the rectangles constructed in the proof of Theorem~\ref{Th:annulus} are embedded.
In the proof of Theorem~\ref{Th:pair}, we will show that all the level curves associated to values in $[0, k_1]$, for some constant $k_{1}<k$,  are simple, piecewise linear closed  curves in either the (free) homotopy class of $E_2^1$ or the class determined by $E_2^2$.
However, for the value $k_1$, it will be proved that the (unique) associated level curve is a figure eight. Furthermore, any level curve of ${\bar g}$ for a value which is larger than $k_1$ and smaller than or equal to $k$, is a simple closed curve in the (free) homotopy class determined $E_1$.
The basic idea of Theorem~\ref{Th:pair} is to cut 
$(\Omega,\bord\Omega,{\mathcal T})$ along the (projection of) a figure eight curve, tile each cylinder according to
Theorem~\ref{Th:annulus}, and glue back. 

We will often work with a series of modified boundary value problems. Each is a slight modification of the initial problem. This important feature of the proofs is due to the fact that the level sets of the original boundary value problem (defined on ${\mathcal T}^{(0)}$) intersect ${\mathcal T}^{(1)}$ in a set which is much larger than ${\mathcal T}^{(0)}$ (see \S2 for details).

In fact, once Theorem~\ref{Th:annulus} is proved, we proceed to prove Theorem~\ref{Th:ladder} by an inductive process. Unlike common proofs in the theory of surfaces, in which a surface is cut along closed geodesics yielding a  pair of pants decomposition, we keep cutting our region along particular (projection of) singular level curves until we encounter a planar pair of pants or an annulus. This is a subtle point, arising from the fact that our gluing needs to preserve lengths of curves determined by two boundary value problems (see Definition~\ref{eq:gradient metric}).
 A technical point (which will be addressed in \S 4) is that one boundary component of one of the pieces (or more) in the decomposition will be singular, hence Theorem~\ref{Th:annulus} cannot be applied directly.
 
The paper is organized as follows. In \S1 we introduce notation, recall a few useful facts concerning boundary value problems on graphs, and define a new natural metric induced by a solution of a boundary value problem (see Definition~\ref{de:gradientmetricl1}). In \S 2  we carry out analysis of level curves of the D-BVP solution; our study brings analysis and topology together in order to provide a good notion of length for level curves as well as a topological description of them.  In \S 3,  we prove Theorem~\ref{Th:annulus}, and  \S 4 is devoted to the proofs of Theorem~\ref{Th:pair} and Theorem~\ref{Th:ladder}. Figures~\ref{fig-4trian}--\ref{fig-4surface},
Figure~\ref{fig-annulus}, and Figure~\ref{fig-5levels}  were generated by two lengthy programs written in Mathematica (version 7.0) by the author and will be available upon request. 
 
\begin{Rem}
The assertions of Theorem~\ref{Th:annulus} may (in principle) be obtained by employing 
 techniques introduced in the famous  paper by Brooks, Smith, Stone and Tutte (\cite{BSST}), in which they study square tilings of rectangles. They define a correspondence between square tilings of rectangles and planar multigraphs endowed with two poles, a source, and a sink. They view the multigraph as a network of resistors in which current is flowing. In their correspondence, a vertex corresponds to a connected component of the union of the horizontal edges of the squares in the tiling; one edge appears between two such vertices for each square whose horizontal edges lie in the corresponding connected components. Their approach is based on {\it Kirckhhoff's circuit laws} that are widely used in the field of electrical engineering. We found the sketch of the proof of Theorem~\ref{Th:annulus} given in \cite{BSST} hard to follow. In fact, another proof of a slight generalization of this theorem was given by Benjamini and Schramm (\cite{BeSch1}, see also \cite{BeSch2} for a related study) using techniques from probability and the dual graph of ${\mathcal T}$.  It is interesting to recall that it was Dehn \cite[1903]{D} who was the first to show a relation between square tiling and electrical networks. 
In an elegant paper, combining a mixture of geometry and probability, Kenyon (\cite{Ke}) used the fact that a  resistor network corresponds to a reversible Markov chain. He showed a correspondence between planar non-reversible Markov chains and trapezoid tilings. A completely different method, for the case of tiling a rectangle by squares, was given using {\it extremal length} arguments in \cite{Sch} by Schramm. One should also mention that Cannon, Floyd and Parry (see \cite{CaFlPa}), using extremal length arguments (similar to these in \cite{Sch}), provide another proof for the existence of tiling by squares. Both \cite{Sch} and \cite{CaFlPa} are  widely known as ``a finite Riemann mapping theorem" and serve as the first step in Cannon's combinatorial Riemann mapping theorem (\cite{Ca}).  We include our proof of Theorem~\ref{Th:annulus}, which is guided by similar principles to some of the ones mentioned above, yet significantly different in a few  points,  in order to make this paper self-contained. In addition, the important work of Bendito, Carmona and Encinas (see for example \cite{BeCaEn1},\cite{BeCaEn2},\cite{BeCaEn3}) on boundary value problems on graphs
 allows us to use a unified framework to even more general problems. Their work is essential to our applications and we will use parts of it quite frequently in this paper as well as in its sequels (\cite{Her2},\cite{Her3}). 
\end{Rem}   

\bigskip

\noindent{\small \ack  {\small \ 
Part of this research was conducted
while the author visited the Department of Mathematics at Princeton University 
during the Spring of 2005. We express our deepest
gratitude for their generous hospitality and inspiration. We
are grateful to G\'{e}rard Besson, Francis Bonahon, Gilles Courtois, Dave Gabai, Steven Kerckhoff and Ted Shifrin
  for enjoyable and helpful discussions on the subject of this paper.} The results of this paper and 
  its sequel (\cite{Her2}) were presented at the $G^3$ (New Orleans, January 2007) and at the  Workshop on Ergodic Theory and Geometry (Manchester Institute for Mathematical Sciences, April 2008). We deeply thank the organizers for the invitations and well-organized conferences.}

\newpage
\section{Preliminaries - boundary value problems on graphs} \label{se:Pre} 
\label{se:BVP} 
We recall some known facts regarding harmonic functions
 and boundary value problems on networks.
We use the notation of Section 2 in \cite{BeCaEn}. Let
$\Gamma=(V,E,c)$ be a {\it finite network}, that is a simple and
finite connected graph with vertex set $V$ and edge set $E$.
Since in this paper $\Gamma=(V,E,c)$ is induced by $(\Omega,\bord\Omega,{\mathcal T})$, we shall further assume that the graph is planar.
 Each edge $(x,y)\in E$ is
assigned a {\it conductance} $c(x,y)=c(y,x)>0$.
 Let ${\mathcal P}({ V})$
denote the set of non-negative functions on $V$. If $u\in {\mathcal
P}( V)$, its support is given by $S(u)=\{ x \in V: u(x)\neq 0 \}$.
Given $F\subset V$ we denote by $F^{c}$ its complement in
$V$.  Set
${\mathcal P}(F)=\{u\in {\mathcal P}(V):S(u)\subset F\}$.  The set   $\bord F=\{ (x,y)\in E: x\in F,  y\in F^{c} \}$  is called 
the {\it
edge boundary} of $F$ and the set  $\delta F=\{x\in F^{c}: (x,y)\in E\ {\mbox
{\rm for some}}\ y\in F \}$ is called the {\it vertex boundary} of
$F$. Let ${\bar F}=F\bigcup \delta F$ and let $\bar E=\{(x,y)\in
E :x\in F\}$.
Given $F \subset V$, let
${\bar \Gamma}(F)=({\bar F},{\bar E},{\bar c})$ be the network
such that ${\bar c}$ is the restriction of $c$ to ${\bar E}$. 
We say that $x\sim y$ if $(x,y)\in \bar E$. For $x\in \bar F$ let $k(x)$ denote the degree of $x$ (if $x\in\delta(F)$ the neighbors of $x$ are taken
  only  from $F$). 
  
 

The following are discrete analogues of
classical notions in continuous potential theory \cite{Fu}.

\begin{Def}\mbox{ \rm (\bf\cite[Section 3]{BeCaEn1})}
 \label{def:energy}  
 Let $u\in {\mathcal P}({\bar  F})$. 
 Then for $x\in \bar F$, the function $\Delta u(x)=\sum_{y\sim x}c(x,y)\left( u(x)-u(y) \right )$ is called
  the Laplacian of $u$ at $x$, \mbox{\rm(}if $x\in\delta(F)$ the neighbors of $x$ are taken
  only  from $F$\mbox{\rm)} and
 the number
 \begin{equation}
 \label{eq:energy}
E(u)= \sum_{x\in\bar F}\Delta u(x)u(x)=\sum_{(x,y)\in \bar E} c(x,y)( u(x)-u(y)
)^2,
\end{equation}
 is called the {\it Dirichlet energy} of $u$.
A function $u\in {\mathcal P}({\bar F})$ is called harmonic in $F\subset V$ if
$\Delta u(x)=0,$ for all $x\in F$.
\end{Def}

For example when $c(x,y)\equiv 1$, 
 $u$ is harmonic at a vertex $x$ if and only if the
value of $u$ at $x$ is the arithmetic average of the value of $u$
on the neighbors of $x$. 
A fundamental property which we will often use  in the {\it maximum property}, asserting that if $u$ is harmonic on $V'\subset V$, where $V$ is a connected subset of vertices having a connected interior, then $u$ attains its maximum and minimum on the boundary of $V'$ (see \cite[Theorem I.35]{So}).


For $x\in \delta(F)$, let $\{y_1,y_2,\ldots,y_m\}\in F$ be its neighbors enumerated clockwise.
 The {\it normal derivative} (see \cite{ChGrYa}) of $u$ at a point
$x\in \delta F$ with respect to a set $F$ is 
\begin{equation}\frac{\bord u}{\bord n}(F)(x)= \sum_{y\sim x,\
y\in F}c(x,y)  (u(x)-u(y)).
\end{equation} 
\smallskip 

The following
 proposition establishes a discrete version of the first classical {\it Green
identity}. It plays a crucial role in the proof the main theorem in  \cite{Her} and is essential in this paper.

\begin{Thm} \mbox{\rm \bf(\cite[Prop. 3.1]{BeCaEn}) (The first Green
identity)}
\label{pr:Green id} Let $F \subset V$ and $u,v\in {\mathcal P}({\bar
F})$. Then we have that
\begin{equation}
\label{eq:Green}
 \sum_{(x,y)\in {\bar
E}}c(x,y)(u(x)-u(y))(v(x)-v(y))=\sum_{x\in F}\Delta
u(x)v(x)+\sum_{x\in\delta(F)}\frac{\bord u}{\bord n}(F)(x)v(x).
\end{equation}
\end{Thm}

\begin{Rem}
 In \cite{BeCaEn}, a second Green identity is obtained. In this paper we will
   use only the one above. In \cite{BeCaEn3} (see in particular Section 2 and Section 3), a systematic study of 
discrete calculus on $n$-dimensional (uniform) grids of Euclidean $n$-space is provided. Their definition of a tangent space may be adopted to our setting and does not require the notion of directed edges. 
\end{Rem}

Let  $\T$  denote a fixed cellular decomposition of $\Omega \cup \partial\Omega$, and
 let  $F\subset {\mathcal T}^{(0)}$ be given.
  Let  $\{c(x,y)\}_{(x,y)\in \bar{E}}$ be a fixed conductance function, and let $\bar{\Gamma}(F)$ be the associated network.
We are interested in functions that solve a boundary value problem (D-BVP) on $\bar{\Gamma}( F)$. The following definition is based on \cite[Section 3]{BeCaEn} and \cite[Section 4]{BeCaEn2}.
\begin{Def}
\label{de:boundary function}
 Let $k>0$ be a constant.
 A solution of a Dirichlet boundary value problem defined on $\bar{\Gamma}(\Omega)$
 is a function $f \in {\mathcal P}({\Omega})$ such that
 $f$ is harmonic in $F$,
  $f|_{E_{1}}=k$ and 
  $f|_{E_{2}}=0$, for some positive constant $k$.
  \end{Def}
  
\begin{Rem}
The uniqueness and existence of a Dirichlet boundary value
solution is guaranteed by the foundational work in \cite[Section 3]{BeCaEn} and  \cite[Section 4]{BeCaEn2}. In fact, their work provides a detailed framework for a broader class of boundary value problems on finite networks.
\end{Rem}

A {\it metric}  on a finite network is a function $\rho : V\ra [0,\infty)$. In particular, the length of a path is given by integrating $\rho$ along the path (see \cite{Ca} and \cite{Du} for a different definition). 
When $\rho\equiv 1$,  the familiar distance function on $V\times V$ is obtained by setting 
$\mbox{\rm dist}( A, B)= (\sum_{x\in \alpha} 1) -1=  k$, where $\alpha=(x,x_1,\ldots x_k)$ is a path with the smallest possible number 
of vertices among all the paths connecting a vertex in $ A$ and a vertex in $B$. 
 We now define a ``metric" which will be used throughout this paper. 
\begin{Def}
\label{de:gradientmetricl1}
Let $F\subset V$ and let  $f\in {\mathcal P}({\bar F})$. The {\it flux-gradient metric} is defined by
\begin{equation}
\label{eq:gradient metric}
 \rho(x)=     \frac{\bord f}{\bord n}(F)(x), \ \ \mbox{\rm if}\  x\in
    \delta(F).
    \end{equation}

This allows us to  define a notion of length to any subset of the vertex boundary of $F$ by declaring: 
\begin{equation}
\label{eq:length gradient metric}
\mbox{\rm Length}(\delta F)=\big |\sum_{x \in \delta F} \frac{\bord f}{\bord n}(F)(x)\big|.
\end{equation}
\end{Def}


In the applications of this paper, we will use the second part of the definition in order to define length of connected components of level curves of the D-BVP solution. 
In \cite[Definition 3.3]{Her}, we defined a similar metric ($l_2$-gradient metric) proving several length-energy inequalities.

\newpage
\section{topology and geometry of piecewise linear level curve}
\label{se:PL level sets} 

Let ${\mathcal G}$ be a polyhedral surface and consider a function $f: {\mathcal G}^{(0)}\rightarrow \RR\cup\{0\}$ such that two adjacent vertices are given different values. 
Let $v\in {\mathcal G}^{(0)}$, and let $w_1,w_2,\ldots,w_k$ be its $k$ neighbors enumerated counterclockwise. Following \cite[Section 3]{LaVe}, consider the number of sign changes in the sequence $\{ f(w_1)-f(v),f(w_2)-f(v),\ldots,f(w_k)-f(v),f(w_1)-f(v)\}$, which we will denote by  $\mbox{\rm Sgc}_{f}(v)$. The index of $v$ is defined as 
\begin{equation}
\label{eq:index}
\mbox{\rm Ind}_{f}(v)= 1- \frac{\mbox{\rm Sgc}_{f}(v)}{2}.
\end{equation}

\begin{Def}
A vertex whose index is different from zero will be called  singular; otherwise the vertex is regular. A level set which contains at least one singular vertex will be called singular; otherwise the level set is regular.
\end{Def} 

A connection between the combinatorics and the topology is provided by the following theorem, which may be considered as a discrete Hopf-Poincar\'{e} Theorem. 

\begin{Thm} \mbox{\rm (\cite[Theorem 1]{Ba}, \cite[Theorem 2]{LaVe}) ({\bf An index formula)}}
\label{Th:index} We have
\begin{equation}
\label{eq:Euler}
\sum_{v\in {\mathcal G}}\mbox{\rm Ind}_{f}(v)=\chi({\mathcal G}).
\end{equation}
\end{Thm}

\begin{Rem}
\label{re:withboundary}
Due to the topological invariance of $\chi({\mathcal G})$, note that once the equation above is proved for a triangulated polyhedron, it holds (keeping the same definitions for $\mbox{\rm Sgc}_{f}(\cdot)$ and $ \mbox{\rm Ind}_{f}(\cdot)$) for any cellular decomposition of $\chi({\mathcal G})$. Also, while the theorem above is stated and proved for a closed polyhedral surface, it is easy to show that it holds in the case of a surface with boundary, where there are no singular vertices on the boundary (simply by doubling along the boundary).
\end{Rem}

Henceforth, we will keep the notation of Section~\ref{se:Intro} and Section~\ref{se:Pre}.  A key ingredient in our proofs of the theorems stated in the introduction is the ability to define a length for a level curve of $g$. 
 The main difficulty in defining such a quantity (for level curves) is the fact that these are not piecewise linear curves of the initial cellular decomposition; hence we cannot  directly define the weight $\rho$ along them (see Equation~(\ref{eq:gradient metric})). 
 
\medskip
Suppose that $L$ is a fixed, simple, closed  level curve and let ${\mathcal O}_1, {\mathcal O}_2$ be the two distinct connected components of $L$ in $\Omega$ with $L$ being the boundary of both (this follows by using the Jordan curve theorem; see for instance \cite[Theorem III.5.G]{Bi} for an interesting proof). 
We will call one of them, say ${\mathcal O}_1$, an {\it interior domain} if all the vertices which belong to it have $g$-values that are smaller than the $g$-value of $L$. The other domain will be called the {\it exterior domain}. Note that, by the maximum principle, one of ${\mathcal O}_1, {\mathcal O}_2$ must have all of its vertices with $g$-values smaller than $L$.

Let $e\in {\mathcal T}^{(1)}$ and $x=e\cap L$. For $x\not\in {\mathcal T}^{(0)}$, we  have now created two new edges $(x,v)$ and $(u,x)$. We may assume that $v\in {\mathcal O}_1$ and $u\in {\mathcal O}_2$. We now define conductance constants $\tilde c(v,x)$ and $\tilde c(x,u)$ by 

\begin{equation}
\label{eq:prehar}
\tilde c(v,x)= \frac {c(v,u) (g(v)-g(u))}{g(v)-g(x)} \ \   \mbox{\rm and}\ \  \tilde c(u,x) = \frac {c(v,u) (g(u)-g(v))}{g(u)-g(x)}.
\end{equation}

We repeat the process above for any new vertex formed by the intersection of $L$ with ${\mathcal T}^{(1)}$. By adding all the new vertices and edges, as well as the piecewise arcs of $L$ determined by the new vertices, we obtain two cellular decompositions, ${\mathcal T}_{{\mathcal O}_1}$ of  
${\mathcal O}_1$ and  ${\mathcal T}_{{\mathcal O}_2}$ of ${\mathcal O}_2$. Also, two conductance functions are now defined on the one-skeleton of these cellular decompositions by modifying the conductance function for $g$ according to Equation~(\ref{eq:prehar}) (i.e. changes are occurring only on new edges). We will denote these by ${\mathcal C}_{{\mathcal O}_1}$ and ${\mathcal C}_{{\mathcal O}_2}$ respectively (on the arcs of $L$ the conductance is identically zero).

\begin{Def}
\label{def:modifiedsolution}
For $i=1,2$, we let $g_i$ denote the solution of the D-BVP determined on the induced cellular decomposition on ${\mathcal O}_i$ defined by the following:
\begin{enumerate}
\item for any boundary component $E$ of ${\mathcal O}_i$ and for every vertex $x\in E$, $g_i(x)=g(x)$,
\item the conductance function on ${\mathcal O}_i$ is ${\mathcal C}_{{\mathcal O}_i}$, and
\item $g_i$ is harmonic in ${\mathcal O}_i$.
\end{enumerate}
\end{Def}
 
It follows from Equation~(\ref{eq:prehar}), Definition~\ref{def:energy}, and the existence and uniqueness theorems in \cite{BeCaEn} that for $i=1,2$, $g_i$ exists  and that $g_i= g|_{{\mathcal O}_i}$. We may now define a flux-gradient metric for $L$ by using 
Equations~(\ref{eq:gradient metric}) and (\ref{eq:length gradient metric}). However, unlike the situation with the boundary components of $\Omega$, we have two possible choices, i.e. computing the normal  derivative along $L$ with respect to ${\mathcal O}_1$ or with respect to ${\mathcal O}_2$. 
Since in the applications we will cut $\Omega$ along a particular $L$ and wish to glue the resulting pieces together along $L$, these lengths computed with respect to the  flux-gradient metric in each domain should be the same.  The situation when $L$ is not simple is more complicated and will be addressed after a detailed analysis of the topological structure of level curves has been carried out
(see the discussion after Remark~\ref{re:complexity}). This analysis is the main core of this section. The next lemma shows that all level curves are in fact closed.

 \begin{Lem}
 \label{le:simpleisclosed}
 A level curve for the function $g$ is piecewise linear and closed, and each simple cycle of $L$ contains at least
 one boundary component of $\partial\Omega$.
 \end{Lem}
\begin{proof}
The assertions of the lemma are certainly true for the components of $\partial \Omega$. Assume now that $L$ is not one of the these level curves and furthermore that it is not closed. Let $q$ be a boundary point of $L$. Since $g$ is extended linearly on edges and in an affine fashion on triangles and quadrilaterals, $q$ may (a priori) be an interior point of an edge of the cellular decomposition, a vertex, or in the interior of a cell. All of these cases are easily ruled out. It remains to prove that any such level curve contains at least one boundary component of $\partial \Omega$. If this is not the case, $L$ (being a finite, closed,  polygonal planar line) bounds a union of $2$-cells in $\Omega$. This is a violation of the maximum principle for the harmonic function $g_1$ constructed as in Definition~\ref{def:modifiedsolution} 
(${\mathcal O}_1$ being one connected component of the union 
of the $2$-cells  bounded by $L$). 
 \end{proof}
 
\begin{Rem}
Since $g$ is extended in an affine fashion along edges, it is clear that two disjoint level curves corresponding to the same $g$ value are at combinatorial distance which is at least one. 
\end{Rem}

The structure of simple curves in the plane can be quite complicated and in fact is not completely understood. In our applications, we need only analyze  the topological structure of closed curves arising as level curves of the affine extension (over triangles and quadrilaterals) of a harmonic function as defined above. We will henceforth work in the piecewise category.

\begin{figure}[h]
  \begin{minipage}[t]{.45\textwidth}
    \begin{center}  
      \includegraphics[width=2in]{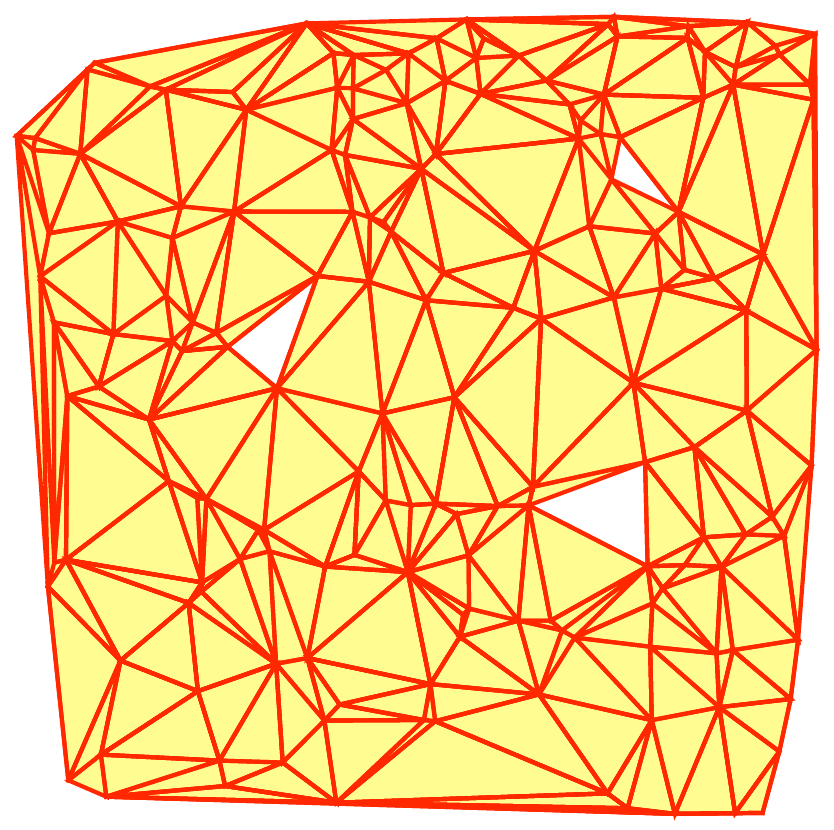}
      \caption{4-connected triangulated domain}
      \label{fig-4trian}
    \end{center}
  \end{minipage}
 \hfill 
  \begin{minipage}[t]{.45\textwidth}
   \begin{center}  
      \includegraphics[width=2in]{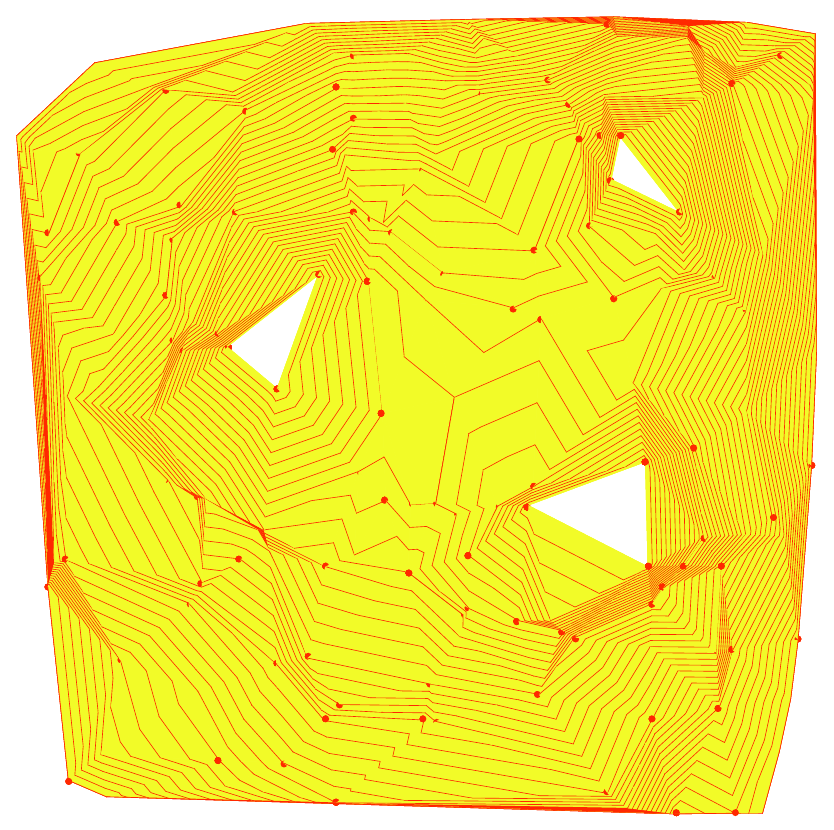}
           \caption{Level curves for the D-BVP}
      \label{fig-4level}
    \end{center}
  \end{minipage}
 \hfill 
 \begin{minipage}[t]{.55\textwidth}
    \begin{center}  
      \includegraphics[width=3in]{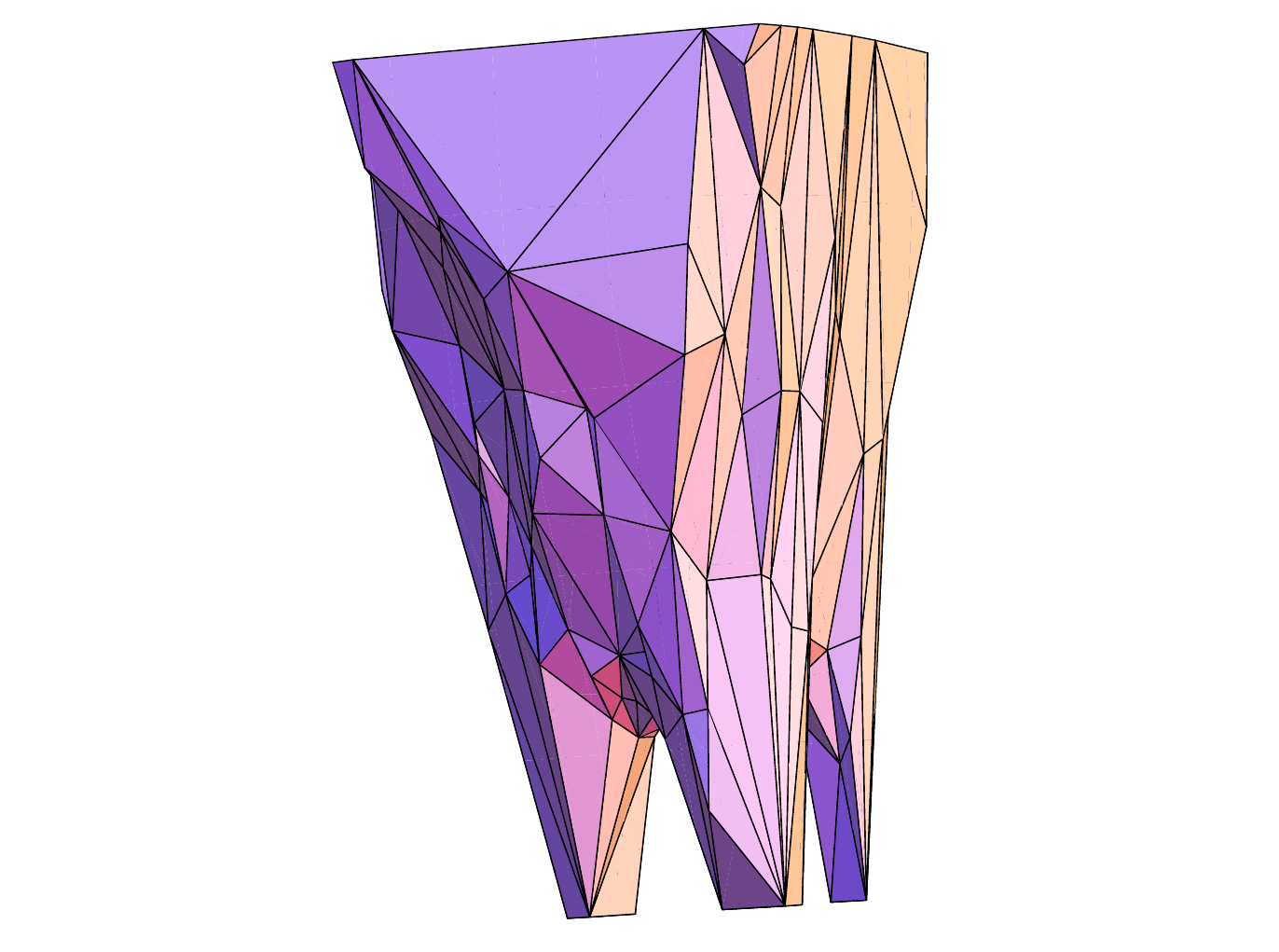}
      \caption{The surface obtained by using $g$ as a height function.}
      \label{fig-4surface}
    \end{center}
  \end{minipage}
  \end{figure}

\medskip

\begin{Def}
\label{de:bouquet}
A generalized piecewise bouquet of circles will denote a union of piecewise bouquets of circles  where the intersection of any two bouquets is at most a vertex. 
Moreover, all such tangencies are exterior, i.e., no circle is contained in the interior
of the bounded component of another. 
\end{Def}

\begin{Pro}
\label{pr:structure}
Two simple cycles which correspond to the same $g$ value either are disjoint or intersect at a single vertex. Furthermore, under these assumptions, a simple cycle is never contained in the annulus bounded by a different one (i.e. a ``tangency" is always external).
\end{Pro}
\begin{proof}
Let $L_1$ and $L_2$ be two given simple cycles corresponding to the $g$-value $l$. Let $A_1$ and $A_2$ be the two (piecewise) annuli which they bound respectively. Assume first that $A_1\cap A_2=\emptyset$ and that  $E=L_1\cap L_2$ is a piecewise arc. Let $I$ and $J$ be the endpoints of $E$.     As mentioned before (see Lemma~\ref{le:simpleisclosed}), both $I$ and $J$ are vertices. The link of the vertex $I$ must contain a two cell (a triangle or a quadrilateral) such that an edge between two of its vertices crosses $E$. This implies that one of these vertices has $g$-value which is larger than $l$ and the other has $g$-value which is smaller than $l$. This is absurd, since by the maximum principle all the vertices in $A_1\cup A_2$ which do not belong to $g^{-1}(l)$ have $g$-values that are strictly smaller than $l$ (none of the vertices of these cells  belong to $g^{-1}(l)$).  
It may also happen that  $L_1\cap L_2= \{P_1,P_2,\ldots, P_k\}$, where $k\geq 1$ is an integer, and where the points $P_i$ are oriented clockwise.  Let $t_2$ be the arc connecting $P_1$ to $P_2$ which belongs to $L_2$ and which does not pass through any other of the $P_I$'s. Let $t_1$ be the arc in $L_1$ which connects $P_1$ to $P_2$ which does not pass through any of the $P_i$'s. Then $t_1\cup t_2$ is a simple closed level curve which bounds an annulus  (see Lemma~\ref{le:simpleisclosed}).
The link of the vertex $P_1$ must contain a triangle or a quadrilateral with one of its edges crossing $t_1$.   As above, this leads to a violation of the maximum principle.

Assume now that $A_1\cap A_2\neq\emptyset$  and (without loss of generality) that $A_1\cup L_1$ contains an arc $t$ of $L_2$. Let its endpoints be $P$ and $Q$ (they both lie on $L_1$). Then $L_1\cup ((L_2\setminus t)\cup E)$ 
consists of two simple cycles that satisfy all the conditions described in the case above; hence, this case may not occur.  

We now consider the case in which one of the annuli is contained in the other. Without loss of generality, assume that $A_1\subset A_2$ and that $E=L_1\cap L_2$ is a piecewise arc. Let $I$ and $J$ be the endpoints of $E$. The link of the vertex $I$ must contain a two-cell, included in $A_2$, where one of its edges connects a vertex in $A_1$ to a vertex  in $A_2\setminus (A_1\cup L_1)$.  This edge crosses $L_1$; hence one of these vertices has $g$-value which is larger than $l$ and the other has $g$-value which is smaller than $l$; this is absurd, since by the maximum principle all the vertices in $A_2$ which do not belong to $g^{-1}(l)$ must have their $g$ value strictly smaller than $l$. It may also be the case that $L_1\cap L_2= \{P_1,P_2,\ldots, P_m\}$ where $m\geq 1$ is an integer and the points $P_i$ are oriented clockwise. Note that we allow all the $P_i$'s to be a single vertex. Consideration of the link of any of the $P_i$'s, and an argument similar to the one above, leads again to a violation of the maximum principle.

Finally, assume that $L_1\cap L_2=\emptyset$. Observe that $A_2\setminus A_1$ is an annulus; both of its boundary curves have $g$ value $l$. By a construction analogous to the one described before
Definition~\ref{def:modifiedsolution}, we obtain a new BVP problem defined on a cellular decomposition of $(A_2\setminus A_1) \cup L_1\cup L_2$. Its solution is harmonic on $A_2\setminus A_1$ and has the same value on the boundary components; hence it is the constant function. The cellular decomposition of the
annulus contains edges (part of edges) or vertices (and perhaps both) of the initial cellular decomposition. Moreover, the values of the solution and $g$ coincide on
$(A_2\setminus A_1) \cup L_1\cup L_2$. Both cases lead to the conclusion that connected pieces of edges of the initial cellular decomposition in the annulus have the same $g$-value, which is absurd.

\end{proof}

\begin{Thm}
\label{th:notsimple}
Let $L$ be a connected level curve for $g$. Then each connected component of $L$ is a generalized piecewise bouquet of circles. 
\end{Thm}

\begin{proof}
First, we  only use the fact that $L$ is a closed polygonal line.
Let ${\mathcal S}$ be the collection of all the self  intersection points of $L$. By definition, edges of the polygonal (closed) curve $L$ may cross only at points from ${\mathcal S}$. 
If ${\mathcal L}$ is simple, the assertion follows immediately. Assume that $L$ is not simple.  
It is well known that we  may express $L$ as a union of simple close polygonal curves. The intersection of any two cycles in this decomposition is a union  of vertices and edges.  
Proposition~\ref{pr:structure} forces restrictions on any such decomposition. In particular, any two simple  cycles in such a decomposition are either disjoint or intersect in a single vertex, none of which is contained in the annuli bounded by the other. The assertion of the theorem follows immediately.
\end{proof}

\begin{figure}[h]
  \begin{minipage}[t]{.45\textwidth}
    \begin{center}  
       \includegraphics[width=2in]{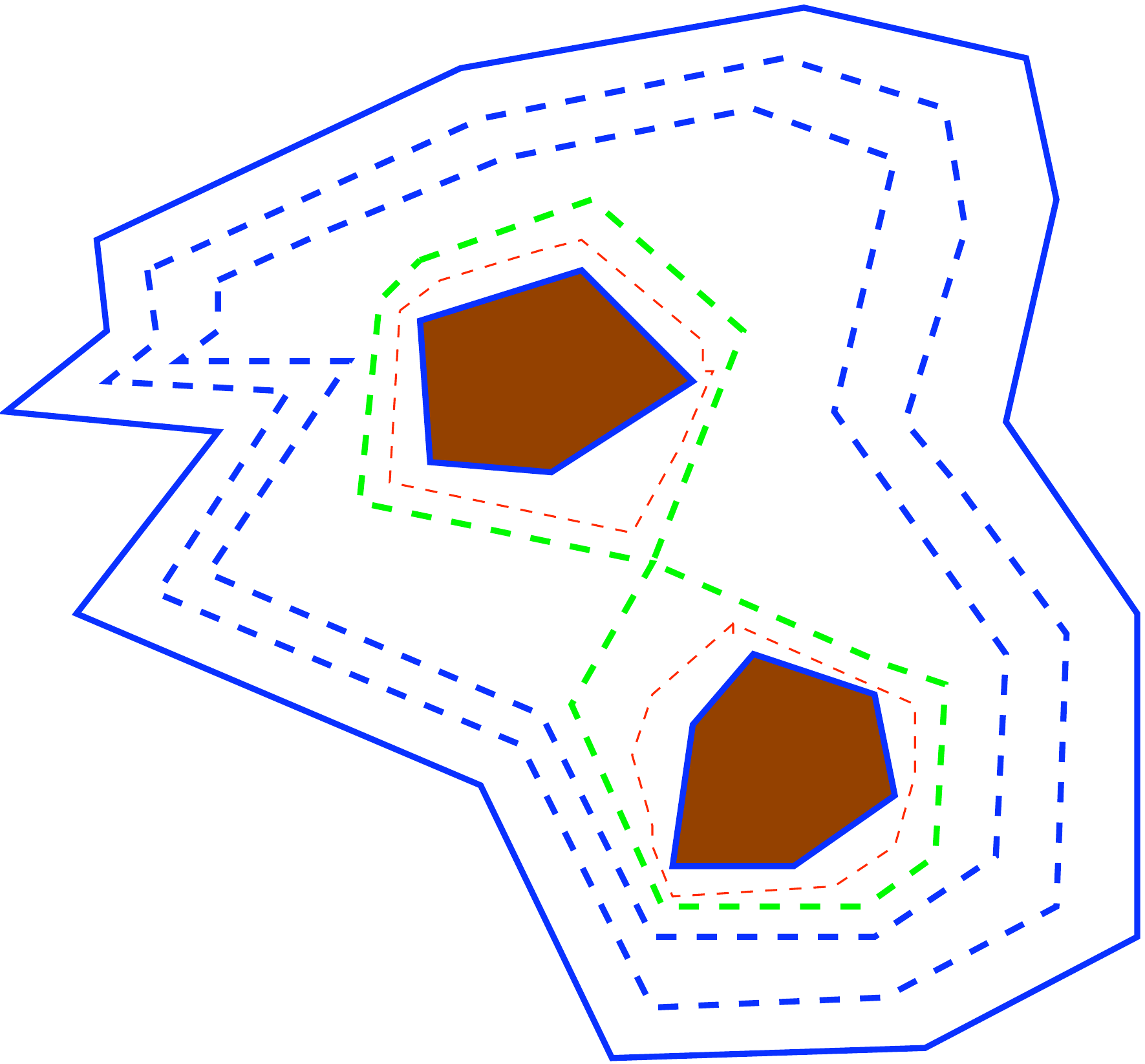} 
      \caption{Level curves in a pair of pants}
      \label{fig-pants}
    \end{center}
  \end{minipage}
 \hfill 
  \begin{minipage}[t]{.45\textwidth}
   \begin{center}  
      \includegraphics[width=2in]{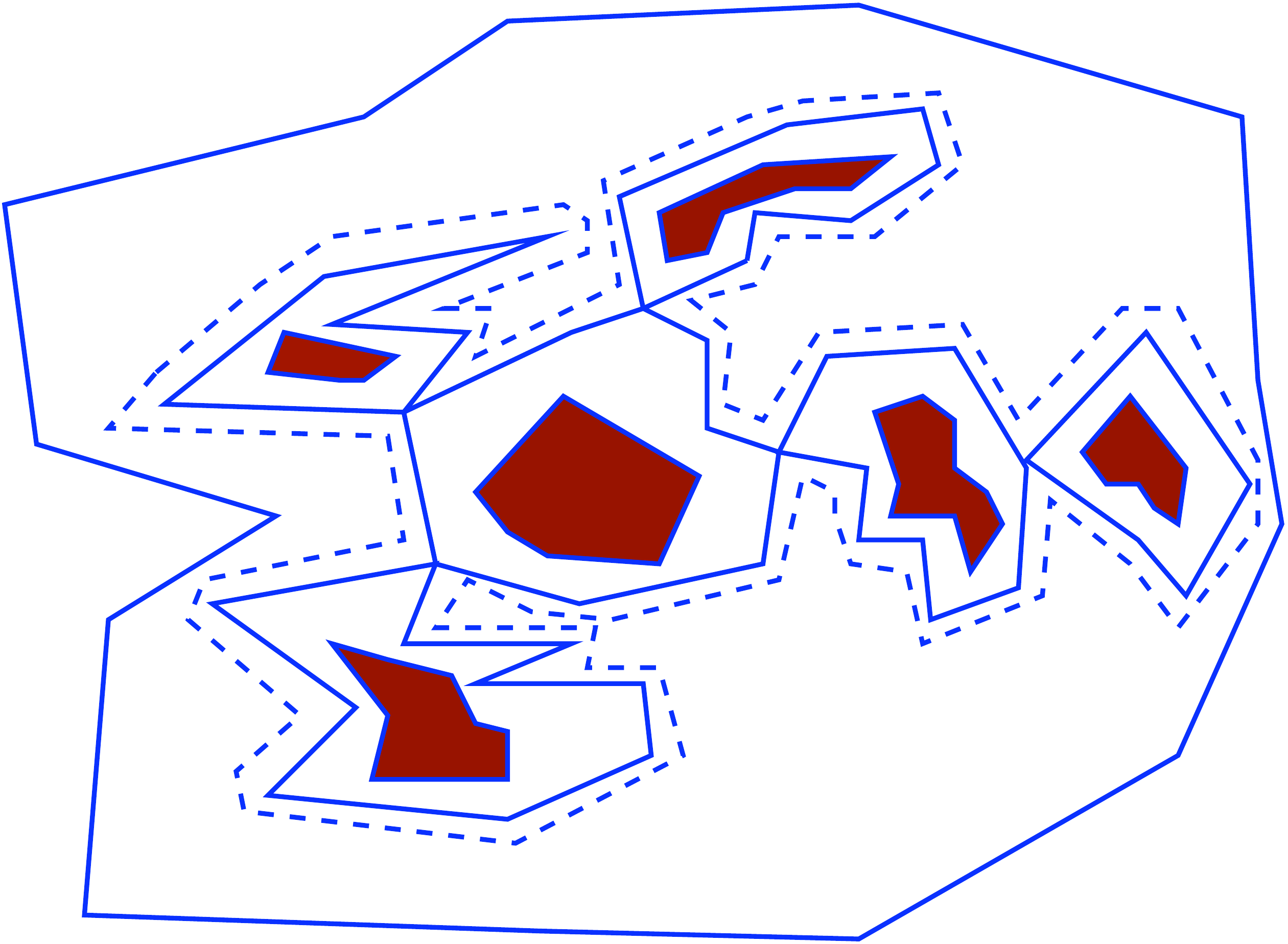}
     \caption{Level curves in a 7 connected domain}
      \label{fig-six}
    \end{center}
  \end{minipage}
 \hfill 
 \begin{minipage}[t]{.55\textwidth}
    \begin{center}  
      \includegraphics[width=2in]{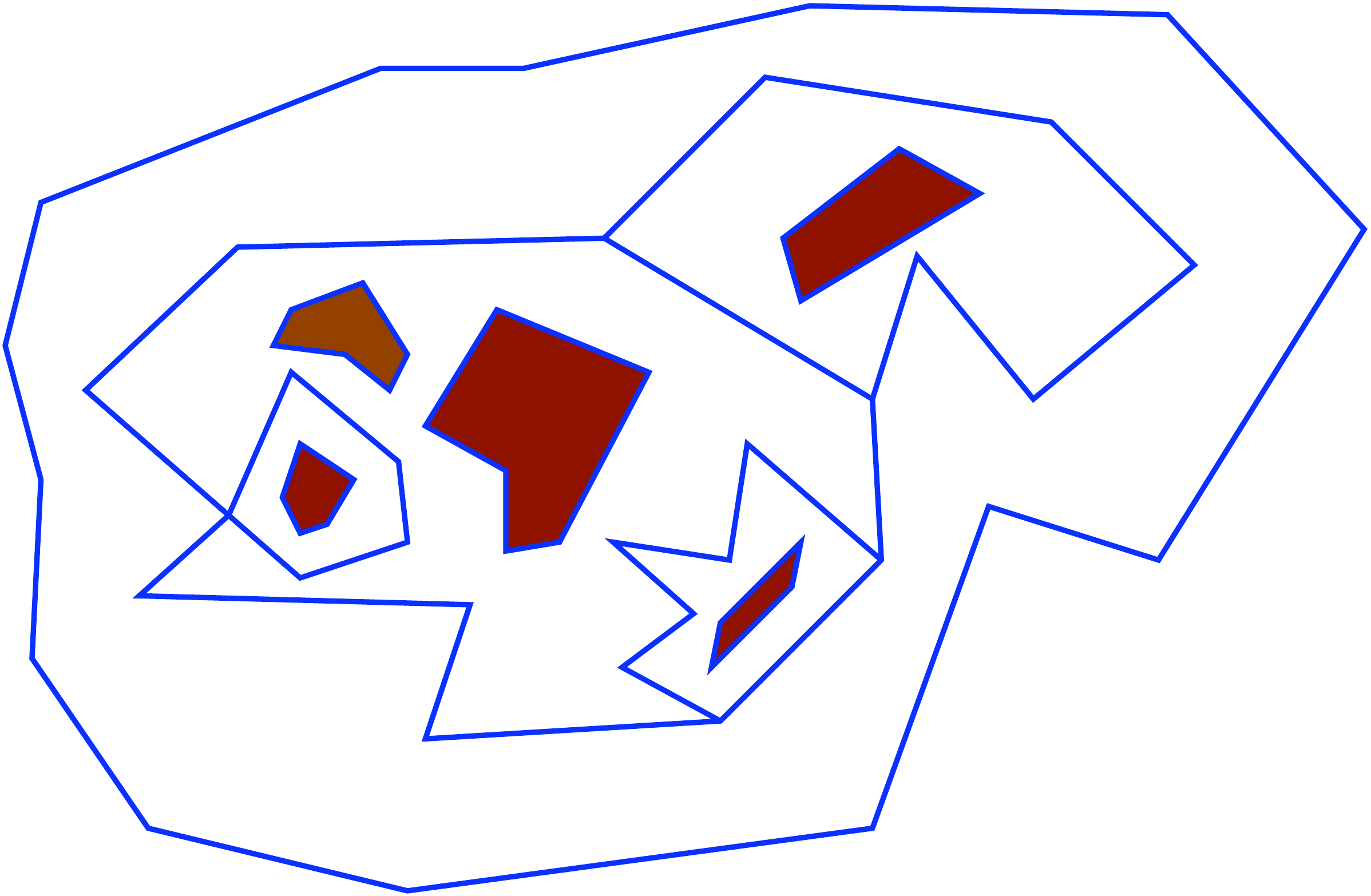}
      \caption{A non-possible level curve in a 6 connected domain}
      \label{fig-sevent}
    \end{center}
  \end{minipage}
  \end{figure}

\begin{Rem}
\label{re:complexity}
Our interest here is in the existence of a decomposition of the level curve into simple pieces. One may also consider the question, or rather a series of questions, regarding a decomposition of the non-simple polygon determined by a general closed polygonal curve. For many deep questions and an excellent survey on this interesting subject (some of the questions involved are known to be {\it NP}), see for instance \cite{Kei}. 
\end{Rem}

The topological structure of a general connected component of a level curve $L$ (Theorem~\ref{th:notsimple}) means that the complement of $L$ in $\Omega \cup \partial\Omega$ is composed of a bounded domain (which is a disjoint union of annuli) in which all vertices have $g$-values smaller than the $g$-value of $L$, and a second domain whose boundary consists of $L$ and possibly other curves in $\partial \Omega$), where $L$ is in the boundary of both. Following a construction similar to the one preceding  
Definition~\ref{def:modifiedsolution},  we may now define two modified conductance functions and two solutions of D-BVP problems.  By abuse of notation, we call the first domain ${\mathcal O}_1$, the second ${\mathcal O}_2$, and we let  ${\mathcal C}_{{\mathcal O}_1}, {\mathcal C}_{{\mathcal O}_2}, g_1, g_2$ be the analogous  quantities. The corresponding two new cellular decompositions and their union will henceforth be called the {\it induced} cellular decomposition.

\begin{Thm}
\label{le:length of pl}
Let $L$ be a level curve for $g$.  Then the following equality holds
\begin{equation}
\label{eq:samelength}
l_{{\mathcal O}_1} (L)= l_{{\mathcal O}_2} (L), 
\end{equation}
where the left-hand side denotes the length of $L$  measured with respect to the flux-gradient metric induced by $g_1$, and the right-hand side denotes the length of $L$  measured with respect to the flux-gradient metric induced by $g_2$.
\end{Thm}

\begin{proof}
First, let us assume that $L$ is simple and corresponds to the $g$-value $h$.
Since $L$ is at combinatorial distance at least one from any other level curves of $g$ of the same value,  the $1$-combinatorial neighborhood of $L$ in ${\mathcal O}_2$ is comprised of vertices that are of $g$-values greater than $h$.

Let $x\in {\mathcal T}^{(0)}\cap L$. Let $y_1,y_2$ be its neighbors in $L$, $\{x_1,\ldots,x_m\}$ in  ${\mathcal O}_1$ and  $\{z_1,\ldots,z_k\}$ in ${\mathcal O}_2$. Since $g$ is harmonic at $x$ we have 
\begin{equation}
\label{eq:haratx}
0=\Delta g(x)=\sum_{y\sim x}c(x,y)\left( g(x)-g(y) \right ).
\end{equation} 
Hence, (since $g(x)=g(y_1)=g(y_2)$)
\begin{equation}
\label{eq:haratx1}
0=\sum_{i=1}^{m}c(x,x_i)\left( g(x)-g(x_i) \right )+\sum_{j=1}^{k}c(x,z_j)\left( g(x)-g(z_j) \right ).
\end{equation} 

Let $x\in{\mathcal T}_{{\mathcal O}_1}\cup {\mathcal T}_{{\mathcal O}_2}$ be a new vertex  in $L$ (that is $x\not \in {\mathcal T}^{(0)}$) with $v\in {\mathcal O}_1$ and $u\in {\mathcal O}_2$. We have

\begin{equation}
\label{eq:newvertex}
\frac{\partial {g_1}}{\partial n}({\mathcal O}_1)(x)=\tilde c(v,x) (g(x)-g(v))=c(u,v) (g(u)-g(v)),
\end{equation}
and
\begin{equation}
\label{eq:newvertex1}
\frac{\partial {g_2}}{\partial n}({\mathcal O}_2)(x)=\tilde c(u,x) (g(x)-g(u))=c(u,v) (g(v)-g(u)).
\end{equation}
By summing the three equations above over all vertices in $L$, the assertion follows.

By combining the arguments above with the topological structure of a general level curve provided by Theorem~\ref{th:notsimple} and a simple modification at singular vertices along $L$, we will now treat the second case. Let us assume that $L$ is not simple. 
Recall that a level curve may intersect itself only at a vertex. Let 
${\mathcal S}=\{v_1,\ldots,v_m\}$ be all such intersection points and $\{\mbox{\rm Ind}_{f}(v_1),\ldots,\mbox{\rm Ind}_{f}(v_m)\}$ be their indices respectively. For a vertex $x\in L$ (of the initial triangulation or the induced one) which is not in ${\mathcal S}$ (note that $\mbox{\rm Ind}_{f}(x)=0$), the arguments above leading to Equation~\ref{eq:haratx} go through precisely in the same manner and yield (with ${\mathcal O}_1,{\mathcal O}_2$ modified as explained immediately after Remark~\ref{re:complexity}) the same conclusion.


Let $v\in {\mathcal S}$. Then $v$ is in the intersection of 
$C_1,\ldots, C_l$ piecewise simple circles where, 
$l=\mbox{\rm Sgc}_{g}(v)/2$ circles. Let $\{x_1^i,\ldots,x_{k(i)}^{i}\}$ be the neighbors of $v$ in ${\mathcal O}_{1}\cap\  \mbox{\rm interior}(C_i)$, $\{y_1, \ldots, y_m\}$  the neighbors of $v$ in $L$ and $\{z_1,\ldots, z_p\}$ in ${\mathcal O}_2$. Since $g$ is harmonic at $v$, we may now modify slightly   Equation~(\ref{eq:haratx1}).  
Since $g(v)=g(y_1)=\ldots=g(y_m)$, we obtain
\begin{equation}
\label{eq:haratx3}
0=\sum_{j=1}^{l}\sum_{i=1}^{k(j)}c(v,x_i^j)\left( g(v)-g(x_{i}^{j}) \right )+\sum_{j=1}^{p}c(v,z_j)\left( g(v)-g(z_j) \right ).
\end{equation} 
The assertion of the theorem now follows by summing over all the vertices in $L$.
Finally, if $L$ has several connected components, the assertion of the theorem follows by addition over each component (which must be at combinatorial distance greater than one from each other). 

\end{proof}

\begin{Rem}
\label{re:adopttoboundary}

We now wish to start using Theorem~\ref{Th:index} and 
Remark~\ref{re:withboundary}. One should note that since boundary components are at the same $g$-level, a modification is needed. This is done in the following way. Along each boundary component, add a piecewise linear curve homotopic to it, and at combinatorial distance which is smaller than one. This can be done in such a way that there is exactly one sign change at each new vertex (not taking into account vertices on the boundary). Let $\tilde{\Omega}\in \Omega$ be the region bounded by these curves.  Doubling along the boundary of $\tilde{\Omega}$ gives a surface without boundary with all of its singular vertices identical to those of $g$, for which the above mentioned theorem may be applied.
\end{Rem}

We end this section with a useful property of  level curves. This property is essential in the proofs of Theorem~\ref{Th:pair} and Theorem~\ref{Th:ladder}. In the course of these proofs, we will need to know that there is a  singular level curve which encloses all of the interior components of $\partial\Omega$. This level curve (which is necessarily singular) is the one along which we will cut the surface. In the inductive process, we keep cutting along such level curves over domains of fewer boundary components until the remaining pieces are annuli.

\begin{Pro}
\label{pr:onethatenclose}
With notation as above, there exists a singular level curve which contains in the   interior of the domain it bounds,  all of the inner boundary components of $\partial\Omega$. Such a level curve is unique. 
\end{Pro}
\begin{proof}
We argue by contradiction.  Throughout the proof, we use the topological structure of level curves  provided by Theorem~\ref{th:notsimple}. Suppose that the first assertion of this proposition is false. Then every singular level curve omits at least one inner boundary component. Let $L$ be such a curve.  Suppose that $L$ omits $k\geq 1$ boundary components and contains $l\geq 2$ boundary components in its interior, where $k+l=m-1$.  Let $\{v_1,v_2,\ldots,v_n\}$ be the singular vertices that belong to $L$; let $\{w_1,w_2,\ldots,w_m\}$ be the singular vertices in the interior of $L$. It is easy to see that (for example by induction)
\begin{equation}
\label{eq:maxindex}
\sum_{x\in \{w_1,\ldots,v_n\}} \mbox{\rm Ind}_g(x) = 1-l. 
\end{equation}
By Theorem~\ref{Th:index} we must have additional singular vertices $\{q_1,\ldots ,q_s\}$ such that 
\begin{equation}
\label{eq:maxindexx}
\sum_{y\in \{q_1,\ldots,q_s\}} \mbox{\rm Ind}_g(y) = (1-k)-1. 
\end{equation}
Let $L_1$ be a singular level curve that has $q_1$ as a singular vertex. Suppose that $L_1$ does not enclose $L$ in its interior. Then 
\begin{equation}
\label{eq:L1}
\mbox{\rm Ind}_g(q_1)\leq (1-k),
\end{equation}
with equality if and only if $L_1$ encloses all $k$ boundary components. Therefore, we must have $s=2$, and a singular level curve that passes through $q_2$ will necessarily enclose all interior components of $\partial \Omega$. Hence, we arrive at a  contradiction.

Assume now that $L_1$ encloses $L$ and $t\geq 1$ out of the $k$ boundary components. Let $\{q_1,q_2,\ldots,q_s\}$ be the singular vertices on the various cycles on $L_1$ which do not belong to the cycle containing $L$ or its interior. 
It is easy to see that 
\begin{equation}
\label{eq:L12}
\sum_{z\in\{q_1,q_2,\ldots,q_s\}}  \mbox{\rm Ind}_g(z)= (1-t).
\end{equation}
By Theorem~\ref{Th:index} we must have additional singular vertices $\{p_1,\ldots, p_n\}$ such that 

\begin{equation}
\label{eq:maxindex1}
\sum_{t\in \{p_1,\ldots,q_n\}} \mbox{\rm Ind}_g(t) =-1-(k-t). 
\end{equation}
We now repeat the process above with $L_1$ replacing $L$, $l+t$ replacing $l$ and $k-t$ replacing $k$. After finitely many times (at most $k-1$), we must end up with a singular level curve with a single singular vertex of index $-2$  that encloses all inner boundary components of $\partial \Omega$.

\end{proof}

\medskip
We are now concerned with geometric properties of the set of lengths of level curves. This is significant for the applications. For example,  even in the case of an annulus (which is the subject of Theorem~\ref{Th:annulus}), it is not clear that  level curves in the same homotopy class have the same length. Since we wish to map a given annulus to a straight cylinder, we must verify this property and its suitable generalizations.

\begin{Thm}
\label{th:lengthsareequal}
There exists a finite set of non-negative numbers 
${\mathcal K}$ such that the following holds:
\begin{enumerate}
\item  ${\mathcal K}$ contains $0$, and $k$ is its maximal element, 
\item  ${\mathcal K}$ is monotone decreasing,
\item any level curve corresponding to a $g$-value which does not belong to ${\mathcal K}$ is simple, and
\item any component of a level curve, which corresponds to a $g$-value which is strictly 
between any two values $k_n>k_m$ in $ {\mathcal K}$, such that $k_n -k_m$ is minimal,  and is contained in a unique simple cycle $C_n$ determined by $k_n$, has its $l_1$-length equal to that of $C_n$. 
\item Moreover, the  length of $C_n$ is equal to the length of the component of $g^{-1}(k_m)$ which it encloses. 
\end{enumerate}
\end{Thm}

\begin{proof}
We let ${\mathcal K}$ denote the set of critical values of $g$ union $k$ and $0$; this gives assertion $(1)$. Since 
$\chi(\Omega\cup\partial\Omega)$ is negative and the index of each singular vertex is less than or equal to $-1$, Theorem~\ref{Th:index} implies that the number of singular vertices, and hence the number of corresponding critical $g$-values, is finite. Assertion $(2)$ follows by ordering. Assertion $(3)$ is the content 
of Lemma~\ref{le:simpleisclosed}. We now turn to proving $(4)$. Let $L$ be a connected component of a level curve as described in $(4)$. By the structure theory provided in Theorem~\ref{th:notsimple} and the maximum principle, $L$ is contained in a unique simple cycle, part of the bouquet which composes  $g^{-1}(k_m)$. Let $C_n$ denote this cycle. 
We first assume that the combinatorial distance between $L$ and $C_n$ is greater than or equal to $2$. We apply the construction that precedes Definition~\ref{def:modifiedsolution} to the level curves $C_n$ and $L$ and the annulus ${\mathcal A}_{L,C_n}$ enclosed by them. By abuse of notation, we will keep denoting the modified set of vertices by ${\mathcal T}^{(0)}$.

Let $F=F_{C_n,L}$ denote the subset of ${\mathcal T}^{(0)}\cap {\mathcal A}_{L,C_n}$. Since the combinatorial distance between $C_n$ and $L$ is greater than or equal to $2$, it follows that 

\begin{equation}
\label{eq:newset}
F\neq\emptyset\  \mbox{\rm  and that}\  
\delta(F)=\{v\in {\mathcal T}^{(0)} : v\in C_n\cup L\}.
\end{equation}

Recall that $g_1$, the modified solution of the D-BVP, is defined on $F\cup\delta(F)$ by requiring that $g_1$ is harmonic (with respect to the modified conductance function) in $F$, $g_1|_{L}=g|_{L}$ and $g_1|_{C_n}= g|_{C_n}$. By the uniqueness of the D-BVP solution, it is clear that $g_1= g|_{F}$.  Let $F^{\circ}$ denote the set of
vertices in $F$ which do not belong to $\delta(F)$. We apply Proposition~\ref{pr:Green id} (the first Green Identity) to $F$ with  $u=g_1$ and $v\equiv 1$.  Equation~(\ref{eq:Green}) then gives
\begin{equation}
\label{eq:Omega1}
0=\sum_{x\in \delta(F)}\frac{\bord g_1}{\bord n}(F^{\circ})(x).
\end{equation}
Hence 
\begin{equation}
\label{eq:Omega1boundary}
 0=\sum_{x\in C_{n}}\frac{\bord g_1}{\bord n}(F^{\circ})(x)+ \sum_{x\in L}\frac{\bord g_1}{\bord n}(F^{\circ})(x),
\end{equation}
which implies that 
\begin{equation}
\label{eq:Omega1boundary1}
 \sum_{x\in L}\frac{\bord g_1}{\bord n}(F^{\circ})(x)= 
 - \sum_{x\in C_{n}}\frac{\bord g_1}{\bord n}(F^{\circ})(x).
\end{equation}
It follows that, with respect to the flux-gradient metric (see 
Equation~(\ref{eq:length gradient metric})), $L$ and $C_n$ have the same length. We now deal with the case that the combinatorial distance between $C_n$ and $L$ is less than 2. While we wish to use again the first Green Identity, care is needed since $F$ is empty in this case. We add a vertex on each edge in the modified cellular decomposition that has a vertex $v$ on $L$ and a vertex $u$ on $C_n$. The value of $g_1(t)$ is defined by the value of $g_1$ on this edge and we also let
\begin{equation}
\label{eq:addingvertices}
\tilde{\tilde{c}}(v,t)= \frac {\tilde{c}(v,t) (g_1(v)-g_1(u))}{g_1(v)-g_1(t)} \ \   \mbox{\rm and}\ \  \tilde{\tilde{c}}(u,t) = \frac {\tilde{c}(u,t) (g_1(u)-g_1(v))}{g_1(u)-g_1(t)}, 
\end{equation}
where $\tilde{c}(v,t)$ and $\tilde{c}(u,t)$ are defined as in equation~(\ref{eq:prehar}) with $t$ replacing $x$. We keep all conductance functions on other edges unchanged.

By applying again the existence and uniqueness theorems in \cite{BeCaEn}, we obtain a solution $g_2$ of a new D-BVP defined on ${\mathcal A}_{L,C_n}$ by requiring that $g_2|_{L}=g_1|_L$, $g_2|_{C_{n}}=g_{1}|_{C_n}$ and that  $g_2$ is harmonic in ${\mathcal A}_{L,C_n}$. It follows that
$g_2$ is $g_1$ on all vertices in ${\mathcal T}^{(0)}$ and is modified so as to have the values of $g_1$ on all vertices defined above. It is easy to verify (using Equation~\ref{eq:prehar} and Equation~\ref{eq:addingvertices})  that 
\begin{equation}
\label{eq:modifiedagain1}
\frac{\partial g_2}{\partial n}(A_{L,C_n})(v)=\frac{\partial g_1}{\partial n}(A_{L,C_n})(v)\  \mbox{\rm and}\ 
 \frac{\partial g_2}{\partial n}(A_{L,C_n})(u)=\frac{\partial g_1}{\partial n}(A_{L,C_n})(u).
\end{equation}
Assertion $(5)$ is proved by following the same techniques that were used in proving  assertion $(4)$ and Theorem~\ref{th:notsimple}. 

\end{proof}

\begin{Rem}
\label{re:lengthof}
\noindent  Observe that once assertion $(5)$ is proved, in order to prove $(4)$, one needs only add at most one vertex on each edge whose vertices lie on $C_n$ and $g^{-1}(k_m)$ respectively.  Also, the same techniques employed in the proof of assertion $(4)$ allow one to get equality between the sum of  the $l_1$-lengths of all the components for a non-singular level curve enclosed in a simple cycle and the $l_1$-length of this cycle. Finally, the case $g^{-1}(0)$ was not stated separately, yet it similarly gives that the sum of the $l_1$-lengths of the inner boundary components of $\partial\Omega$ equals the $l_1$-length of the outer boundary. We will revisit this last point in the proof of Theorem~\ref{Th:pair}.

\end{Rem}

\section{the case of an annulus}
\label{se:Annulus}
In this section we prove Theorem~\ref{Th:annulus}.  The proof consists of two parts. First, we will show that there is a well-defined mapping from ${\mathcal T}^{(1)}$ into a set of 
 (Euclidean) rectangles embedded in the cylinder $S_{\mathcal A}$. The crux of this part is the fact  that level curves for $g$ have the same induced length (measured with the  $l_1$ metric, see Remark~\ref{re:lengthof}) and  a simple application of the maximum principle. Second, we will show that the collection of these rectangles forms a tiling of $S_{\mathcal A}$ with no gaps and no overlaps. This is a consequence of the first part and an energy-area computation.  The latter follows by our construction, the dimensions of $S_{\mathcal A}$, and the first Green identity (see Theorem~\ref{pr:Green id}).

\begin{figure}[h]
 \begin{minipage}[t]{.45\textwidth}
    \begin{center}  
      \includegraphics[width=2in]{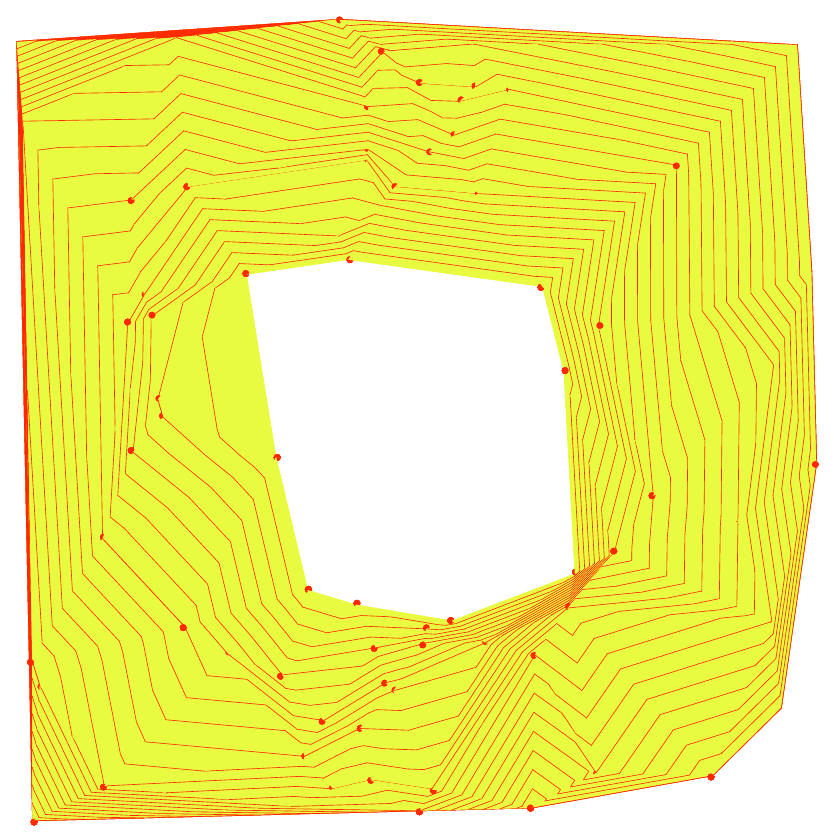}
            \caption{Level curves in an annulus}
      \label{fig-annulus}
    \end{center}
  \end{minipage}
\end{figure}

Given a straight Euclidean cylinder of height $h$, we will endow it with coordinates arising naturally from $\SS^{1}\times[0,h]$. The boundary component corresponding to $h$ will be called the top and the other one will be called the bottom. 
Before providing  the proof, we need a definition which will simplify keeping track of the mapping $f$. 
\begin{Def}
\label{de:marker}
A marker on a straight Euclidean cylinder  is a vertical closed interval  which is the isometric image of  $\theta\times [a,b]$, for some $\theta \in [0,2 \pi)$ and $[a,b]\subset [0,h]$ with $a<b$. The marker's  uppermost end-point corresponds to $(\theta,b)$ and its lowest end-point to $(\theta,a)$.
\end{Def}

\noindent{\bf Proof of Theorem~\ref{Th:annulus}.}
Let $S_{\mathcal A}$ be a straight Euclidean cylinder with height $H=k$  and 
circumference $$C=\sum_{x\in E_1}\frac{\bord g}{\bord n}(x). $$
Let ${\mathcal L}=\{L_1,\ldots,L_k\}$ be the level sets for $g$ corresponding to the vertices in ${\mathcal T}^{(0)}$ arranged in a descending $g$-values order.  We place a vertex at each intersection of an edge with an $L_i,\   i=1,\ldots, k,$, and if necessary more vertices on edges so that any two successive level curves in ${\mathcal L}$ are at combinatorial distance (at least) two. We call the first group of added vertices type I and the second type II (recall that conductance along edges are changed as well, according to the discussion preceding Definition~\ref{def:modifiedsolution}). In particular, the assertions of Theorem~\ref{th:lengthsareequal} and 
Remark~\ref{re:lengthof} hold. Since $\chi({\mathcal A})=0$ and the index of a singular vertex is always negative, the index formula (Theorem ~\ref{Th:index}) prevents the existence of singular vertices in ${\mathcal A}$.  Therefore, ${\mathcal K}=\{0,k\}$,  and furthermore the length of any $g$-level curve is equal to $C$.

Let the vertices $\{x_1,\ldots,x_p\}$ in $E_1=L_1$ be ordered  counterclockwise (and on any level curve), 
starting with $x_1$. Let $\{y_1,y_2,\ldots, y_t\}$ be its type I neighbors in the induced cellular decomposition oriented clockwise (which will always be the ordering for neighbors). 
We identify $x_1$ with $0\times k$ in the coordinates above. We associate markers   $\{m_{x_1,y_1},\ldots,m_{x_1,y_t}\}$ with $x_1$ in the following way. The length of marker $m_{x_1,y_s}$, for $s=1,\ldots,t$, is equal to (the constant) $g(x_1)-g(y_s)$ and its uppermost end-point is positioned at the top of $S_{\mathcal A}$ at 

\begin{equation}
\label{eq;positionmarker}
x_1 +
\frac{\sum_{k=1}^{s-1} c(x_1,y_k)(g(x_1)-g(y_k))}{\frac{\partial g}{\partial n}({\mathcal A})(x_1)}\times 2\pi,
\end{equation} 
measured counterclockwise.  For each edge $e_{u,v}=[u,v]$ with $g(u)>g(v)$, let $Q_{u,v}$ be a Euclidean rectangle with height equals to $g(u)-g(v)$ and width equals to $c(u,v)(g(u)-g(v))$. We will identify a planar straight Euclidean rectangle and  its image, under an isometry, in a straight Euclidean cylinder. 
For $s=1,\ldots,t$,
 we position $Q_{x_1,y_s}$ in $S_{\mathcal A}$ in such a way that its leftmost edge coincides with $m_{x_1,y_s}$. By construction and the position of the markers,
 
\begin{equation}
\label{eq;intersectionis1}
Q_{x_1,y_s}\cap Q_{x_1,y_{s+1}}= m_{x_1,y_{s+1}}.
\end{equation}
Assume that we have placed markers and  rectangles associated to all the vertices up to $x_k$  where $k<p$; let $z_1$ be the leftmost neighbor of $x_{k+1}$ and let $Q_{x_{k},v}$ be the rightmost rectangle associated with $x_k$. We now position the marker $m_{x_{k+1},z_1}$, associated with $x_{k+1}$ and $z_1$ 
so that it is lined with the rightmost edge of $Q_{x_{k},v}$ and his upper end-point is at the top of $S_{\mathcal A}$.  We continue  placing markers and rectangles corresponding to the rest of the neighbors of $x_{k+1}$. We terminate these steps when $k=p$. Note that the top of $S_{\mathcal A}$ is completely covered by the top of the rectangles constructed above where intersections between any of these (top edges)  is either a vertex or empty. 

For all $1<n<k$, assume that all the markers corresponding to vertices in $L_{n-1}$ and their associated rectangles have been placed as above in such a way that the following condition, which we call {\it consistent},  
  holds. For $[w,v]\in T^{(1)}$ with $g(w)>g(v)$ and $s\in [w,v]$ a vertex of type I, the uppermost end-point of the marker $m_{s,v}$ coincides with the lowest end-point of the marker $m_{w,s}$; moreover, the two rectangles $Q_{w,s}, Q_{s,v}$ tile $Q_{v,w}$.  
  Informally, this condition allows us to ``continuously extend" rectangles associated with edges that cross level curves along these, and therefore will show that that edges in ${\mathcal T}^{(1)}$ are mapped (perhaps in several steps) onto a unique rectangle.

We will now show how to place the markers and rectangles corresponding to the vertices of the level set $L_{n}$ in a consistent way. The uppermost end-point  of each marker associated with a vertex $v\in L_n$ in this level set, and any of its neighbors in $L_{n+1}$ is placed in $S_{\mathcal A}$ at height $g(v)$. Observe that $v$ is a vertex in some  $[q_i,v]$, where $q_i$ belongs to $L_n$. Choose among all such edges the rightmost (viewed from $L_n$). Let $[q_0,v]$ be this edge and let $m_{q_{0},v}$ be its marker. Place the marker of $v$ which corresponds to an edge $[v,w]$ with $w$ being the leftmost vertex among the neighbors of $v$ in $L_{n+1}$, so that its uppermost end-point coincides with the lowest end-point  of $m_{q_0,v}$. Let $s_0$ be a vertex of type I on $L_n$. By definition, $s_0$ is connected to a unique vertex $v_0\in L_{n+1}$ and to a unique vertex $w_{0}\in L_{n-1}$. Let $s_q$ be the first vertex to the left of $s_0$ of type I in $L_n$. By definition, $s_q$ is connected to a unique vertex $w_p$ in $L_{n-1}$ and to a unique vertex $v_l \in L_{n+1}$. Let $\{s_1,\ldots,s_{q-1}\}$ be the vertices on $L_n$ between $s_0$ and $s_q$ and let $\{w_1,\ldots, w_{p-1}\}$ be the vertices on $L_{n-1}$ between $w_0$ and $w_p$. 
Let ${\mathcal Q_1}={\mathcal Q}_{w_0,s_0,s_q,w_p}$ be   
the piecewise linear rectangle enclosed by $[w_0,s_0]\cup [w_p,s_q]\cup L_{n-1}\cup L_{n}$, and which contains $\{s_0,\ldots,s_q\}$, and let ${\mathcal Q_2}={\mathcal Q}_{s_0,v_0,s_q,v_l}$ be  
the piecewise linear rectangle enclosed by $[s_0,v_0]\cup [s_q,v_l]\cup L_{n}\cup L_{n}$, and which contains $\{s_0,\ldots,s_q\}$.

 In order to prove that the consistent condition holds for all markers and rectangles created in this step, it suffices to prove it at $s_q$, assuming (with out loss of generality)  that the first marker we placed in a consistent way is $m_{s_0,v_0}$. 
By construction (see in particular Equation~(\ref{eq;positionmarker})) we need to prove that
\begin{eqnarray}
\label{eq:consistent}
\ \ \sum_{i=1}^{p-1}\frac{\partial g}{\partial n}({\mathcal Q}_1)(w_i)+\frac{\partial g}{\partial n}({\mathcal Q}_1)_{\mbox {\rm {\small right}}}(w_0)
& =&\\  \sum_{i=1}^{q-1}\frac{\partial g}{\partial n}({\mathcal Q}_2)(s_i)+\frac{\partial g}{\partial n}({\mathcal Q}_2)_{\mbox {\rm right}}(s_0),\nonumber
\end{eqnarray}
where the subscript ``right"  indicates that neighbors in the expressions  are taken from  ${\mathcal Q_1}$ (first line) or ${\mathcal Q_2}$ (second line) only. It is easy to check that since the modification of $g$ is harmonic at each $s_i, i=0,\ldots q$, and since $s_0$ and $s_q$ are type I vertices, Equation~(\ref{eq:consistent}) holds.

To conclude the construction, continue as above, exhausting all all vertices in $L_{n}$.
By the maximum principle, our construction, and the fact that all level curves have their  flux-gradient length equal to C, it is clear that the union of the rectangles is contained in $S_{\mathcal A}$.

\begin{figure}[htbp]
\begin{center}
 
\includegraphics[width=4in]{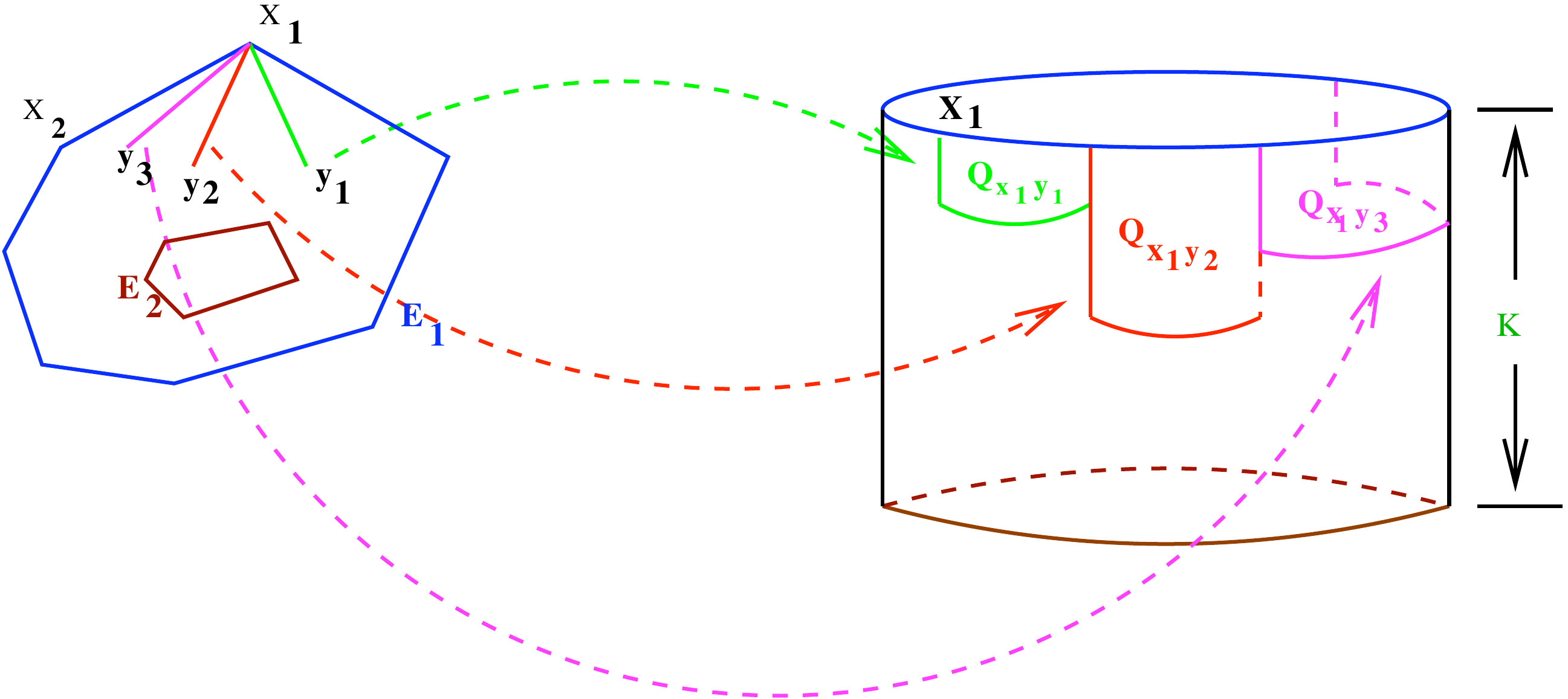}
 
\caption{Several markers and rectangles on the Euclidean cylinder after the completion of the construction.}
\label{figure:cylinder}
\end{center}
\end{figure}

We now prove that the collection of rectangles constructed above tiles $S_{\mathcal A}$ leaving no gaps. Without loss of generality, suppose that the collection of rectangles does not cover a strip of the form
$[\theta_1,\theta_2]\times [h_0,h_1]$ in $S_{\mathcal A}$, where $0\leq\theta_1<\theta_2\leq 2\pi$ and $0\leq h_0<h_1\leq k$.  By harmonicity, there exists at least one path whose vertices belong to ${\mathcal T}^{(0)}$ such that the values of $g$ along this path are strictly decreasing from $k$ to $0$.

In particular, the value $h_1$ is attained on some edge or vertex of this path. By construction, a gap in a $g$-level curve (i.e. an arc of level curve which is not covered by the top edges of rectangles) never occurs when $h_1$ is the $g$-value associated to a vertex in the modified cellular decomposition. Hence, we may assume that $h_1$ is attained in the interior of an edge. Let $L_{h_1}$ be the corresponding level curve. Recall that $L_{h_1}$ is simple and closed (as are all other level curves of $g$; see Lemma~\ref{le:simpleisclosed}).
We now follow the construction preceding Definition~\ref{def:modifiedsolution} and let 
$\{u_1,u_2,\ldots, u_q\}$ be all the new vertices on $L_{h_1}$, that is we place a vertex at each intersection of an edge in $\tilde{\mathcal T}^{(1)}$ with $L_{h_1}$.  As mentioned in the first paragraph of the proof, the flux-gradient length of $L_{h_1}$ is equal to C. Moreover, this length is
equal to
\begin{equation}
\label{eq:lengthofgap}
\sum_{i=1}^{q} \frac{\partial g}{\partial n}({\mathcal  O_1})(u_i),
\end{equation}
where ${\mathcal O}_1$ is the interior of the annulus enclosed by $L_{h_1}$ and $E_2$ (see the discussion proceeding Definition~\ref{def:modifiedsolution}). 
In particular, we may now place markers and rectangles associated to the collection of vertices $\{u_1,u_2,\ldots, u_q\}$ so that $L_{h_1}$ is covered by the top edges of these rectangles. Recall that up to this moment, markers and rectangles were associated only to vertices of type I in addition to those in ${\mathcal T}^{(0)}$. It is very easy to check that this can be done in a consistent way (using Equation~(\ref{eq:prehar})). Since $g$ is extended affinely over edges, every value between $h_0$ and $h_1$ is attained by $g$. Repeating this argument shows that all level curves are covered by rectangles. Hence the collection of rectangles leaves no gaps in $S_{\mathcal A}$.

Using an area argument, we now finish the proof by showing that there is no overlap between any two of the rectangles. Let ${\mathcal R}$ be the union of all the rectangles. By definition,
\begin{equation}
\label{eq;arecomputation}
\mbox{\rm Area}{(\mathcal R)}=\sum_{[x,y]\in \tilde{\mathcal T}^{(1)}}  c(x,y)(g(x)-g(y))(g(x)-g(y)).
\end{equation}
Note that the sum appearing in the right-hand side of Equation~(\ref{eq;arecomputation}) is computed over the induced cellular decomposition. A simple computation (using again Equation~(\ref{eq:prehar})) and the fact that the construction is consistent shows that this sum is equal to the one taken over $[x,y]\in {\mathcal T}^{(1)}$. Hence, the right-hand side of this equation is the energy $E(g)$ (see Definition~\ref{def:energy}). Therefore, by the first Green identity, applied with $u=v=g$ (see Theorem~\ref{pr:Green id}), and the dimensions of $S_{\mathcal A}$,  we have 
\begin{equation}
\label{eq:areasareequal}
E(g)=\mbox{\rm Area}{(\mathcal R)}=\mbox{\rm Area}(S_{\mathcal A}).
\end{equation}
Hence, since the rectangles do not overlap, they must tile $S_{\mathcal A}$. It is also evident that the mapping $f$ constructed above is energy preserving. 

\hfill$\Box$


\section{the case of an $m$-connected bounded planar region, $m>2$}
\label{se:pop}
There is a technical issue which we need to address in the proofs of Theorem~\ref{Th:ladder} and Theorem~\ref{Th:pair}. Since Proposition~\ref{pr:Green id} is applied all along, we must make sure for example, that the combinatorial distance between any two distinct level curves is at least two.  Therefore, one adds as many vertices of type I and type II as needed,  and modifiy conductances accordingly (as done at the beginning of the proof of Theorem~\ref{Th:annulus}) until the assertions of Theorem~\ref{th:lengthsareequal} and Remark~\ref{re:lengthof} hold.

\subsection{Pair of pants}
\label{subse:pants}
In this subsection we provide a proof of Theorem~\ref{Th:pair}. One interesting ingredient  of the proof is the construction of a good mapping  (in the sense of Theorem~\ref{Th:annulus}) from a planar annulus with one singular boundary component into a Euclidean cylinder having one boundary component, a singular curve and the other one which is simple. The construction of the mapping $f$ is then achieved by gluing two Euclidean cylinders to the one above along the singular component. 
\medskip

\noindent{\bf Proof of Theorem~\ref{Th:pair}.}

We extend the solution $g$ to an affine map ${\bar g}: |{\mathcal P}\cup{\bord P}|\rightarrow {\mathbf R}^{+}\cup{0}$ and use ${\bar g}$ as a height function for $|{\mathcal P}\cup{\bord P}|$. This gives a two-dimensional polyhedron, denoted by ${\bar P}_g$, which is  homotopically equivalent to ${\mathcal P}$. We have 
\begin{equation}
\chi({\bar P}_g)=\chi({\mathcal P})=-1.
\end{equation} 

Double ${\bar P}_g$ along its boundary to obtain a genus $2$ closed surface which  will be denoted by 
${\mbox{\rm D}}{\mathcal P}$. The Index Theorem (see Theorem~\ref{Th:index}) asserts that
\begin{equation}
 \sum_{v\in {\mbox{\rm D}}{\mathcal P}} {\mbox {\rm Ind}}_{g}(v)= -2.
\end{equation}

It follows from Equation~(\ref{eq:index}) and  a simple Euler characteristic computation that we must have a single vertex $u\in {\bar P}_g$ whose index equals $-1$, which means that ${\mbox{\rm Sgc}}(u)=4$.  Let $L(u)$ be the figure eight curve corresponding to the value ${\bar g}(u)$.  Then $L(u)=L_{I}\cup L_{II}$, where $L_{I}$ and $L_{II}$ are two simple cycles intersecting at $u$ (using here the assertion of Theorem~\ref{th:notsimple}).  

Moreover, since  $E_2^1$ and $E_2^2$ are the only inner components of $\partial \Omega$, by applying Lemma~\ref{le:simpleisclosed} we may assume that $E_2^1$ is contained in $L_{I}$ and $E_2^2$ is contained in $L_{II}$. By cutting along $L(u)$, we can decompose $\Omega\cup\partial \Omega$ into three components. The first, denoted by $\Omega_1$, has boundary components $E_1$ and $L(u)$. The second, denoted by $\Omega_{2}^1$, has boundary components $E_2^1$ and 
$L_{I}$, and the third, denoted by $\Omega_{2}^2$, has boundary components $E_2^2$ and $L_{II}$. Note that the interiors of these components are all homeomorphic to $S^1\times(0,1)$, and moreover that   $\Omega_{2}^1$ and $\Omega_{2}^2$ are each homeomorphic to an annulus, however $\chi(\Omega_1)=-1$.  Therefore, $\Omega_1$ may be viewed as an annulus with one singular boundary component, or equivalently as a really short pair of pants, i.e. one with no cuffs.  We now define three D-BVPs on the above domains following the discussion preceding Definition~\ref{def:modifiedsolution}. Note that an essential property of $L(u)$, which is given by Theorem~\ref{th:notsimple}, allows us to define a D-BVP on $\Omega_1$.

Let  $g_1$ be the solution on $\Omega_1$, $g_2^1$ on
$\Omega_2^{1}$ and $g_2^2$ on $\Omega_2^2$.  
By Remark~\ref{re:lengthof}, we have that the lengths of $L_{I}$ and $E_2^1$, with respect to the flux-gradient metric induced by $g_2^1$, are both equal to
\begin{equation}
\label{eq:length of LI}
\big |\sum_{x\in E_2^1}\frac{\bord g}{\bord n}(\Omega_2^1)(x)\big |.
\end{equation}

Similarly, we have for $\Omega_2^2$ that the lengths, with respect to the flux-gradient metric induced by $g_2^2$ of $L_{II}$ and $E_2^2$, are equal to
\begin{equation}
\label{eq:length of LII}
\big |\sum_{x\in E_2^2}\frac{\bord g}{\bord n}(\Omega_2^2)(x)\big |.
\end{equation}

 Finally, by applying Theorem~\ref{th:lengthsareequal} with $C_n=E_1$ and $L=L(u)$, we have that the lengths of $E_1$ and $L(u)$, with respect to the $l_1$ gradient metric induced by $g_1$, are equal to 
 \begin{equation}
\label{eq:length of L}
\sum_{x\in E_1}\frac{\bord g}{\bord n}(\Omega_1)(x).
\end{equation}
 
 \begin{figure}[h]
    \begin{center}  
      \includegraphics[width=5in]{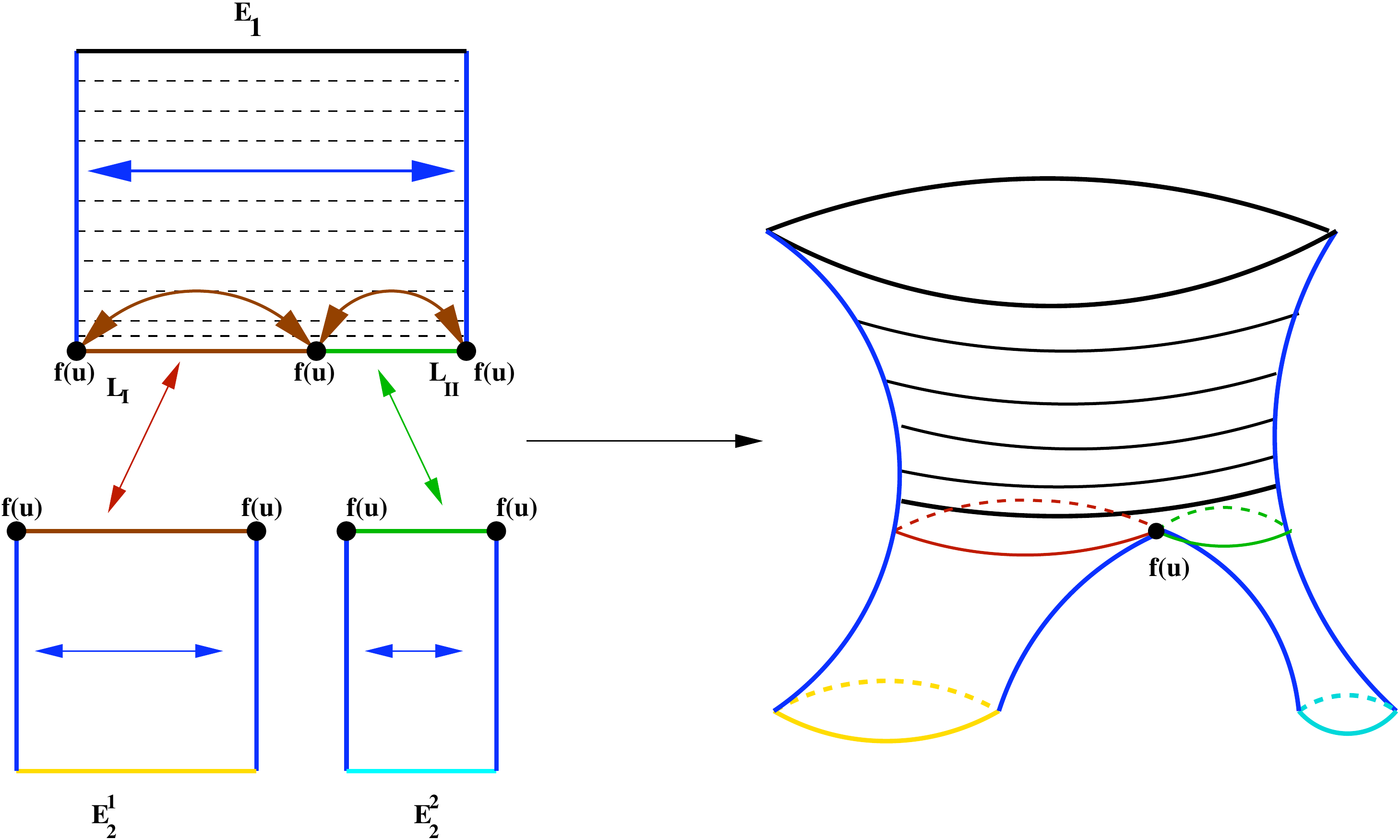}
            \caption{The construction of the Euclidean pair of pants with one singular point.}
      \label{fig-Pair}
    \end{center}
\end{figure}

We now turn to the construction of $(S_{\mathcal P},f)$. First, apply 
Theorem~\ref{Th:annulus} to $\Omega_2^1$ and to $\Omega_2^2$. The outputs are two straight Euclidean cylinders $S_{\Omega_2^1}$ and $S_{\Omega_2^2}$, two mappings $f_{\Omega_2^1}: {\Omega_2^1}\rightarrow{\mathcal S}_1 $ and $f_{\Omega_2^2}:{\Omega_2^2}\rightarrow{\mathcal S}_2$ (where the maps, their domains and the cylinders are described in detail in Theorem~\ref{Th:annulus}). Second, we wish to apply Theorem~\ref{Th:annulus} to $\Omega_1$. However, one boundary component of it ($L(u)$) is singular at $u$.  Note that the construction described in the proof of Theorem~\ref{Th:annulus} will go through for any annulus contained in $\Omega_1$ with one  boundary component being $E_1$ and the other being $L_{g(L(u))+\epsilon}$ (with $\epsilon>0$ arbitrarily small). Thus, we obtain a sequence of Euclidean cylinders that have the same circumference (the flux-gradient metric length of $E_1$)  and their heights converge to 
$k-g(L(u))$. 

As $\epsilon\rightarrow 0$, the level curves corresponding to $g(L(u))+\epsilon$ converge geometrically to $L(u)$. Hence, 
the limiting cylinder ${\mathcal S}_{0}$ is a Euclidean cylinder with one singular boundary component (bottom). 
We call one boundary component of this cylinder {\it top}  and the other {\it bottom}.
Concretely, ${\mathcal S}_{0}$  is obtained by taking a Euclidean cylinder of height $k-g(L(u))$ and circumference which is equal to the flux-gradient metric length of $E_1$, picking two points on one boundary component of it at a distance which equals the flux-gradient metric length of $L_{I}$, and identifying them. We will abuse notation and denote this point by $u$.  Since identification occurs only at the boundary, the tiling persists in the limit. We denote by $f_{\Omega_1}$ the mapping given by Theorem~\ref{Th:annulus}, modified on the bottom as above. Let 
\begin{equation}
\label{eq:mapping}
f=f_{\Omega_1}\cup f_{\Omega_2^1}\cup f_{\Omega_2^2},
\end{equation}
 where the domain of $f$ is $\Omega_1\cup\Omega_2^1\cup\Omega_2^2$  and its image is obtained by gluing isometrically $S_1$ and $S_2$ to the bottom in such a way that the intersection consists of only one point (see Figure~\ref{fig-Pair}), denoted by $f(u)$. By Remark~\ref{re:lengthof}, we have that the length (measured with the $l_1$-induced metric) of the bottom of ${\mathcal S}_{0}$ is equal to the sum of the lengths of the tops of the two cylinders ${\mathcal S}_1$ and ${\mathcal S}_2$ (one can also obtain this by applying Proposition~\ref{pr:Green id} and Equations~(\ref{eq:length of LI})-(\ref{eq:length of L}) directly). Furthermore, we 
 require the gluing to be consistent; i.e  the gluing described above respects rectangles corresponding to edges which cross  $L(u)$ (see the Proof of Theorem~\ref{Th:annulus}). 
 The fact that this can be done may be justified by basically the same argument appearing before and after Equation~(\ref{eq:consistent}). Note that one consequence of this is that $f(u)$ is uniquely determined (the modified $g$ being harmonic being the key issue).

It is easy to check that the only cone angle is at this point and is equal to $4\pi$. The mappings $f_{\Omega_2^1},f_{\Omega_2^2}$ are energy preserving in the sense explained in Theorem~\ref{Th:annulus}; the mapping $f_{\Omega_1}$ is also energy preserving (since the identification is done on bottom only). Since the gluing is by isometries, and the cylinders involved meet only at one vertex, the mapping $f$ is energy preserving. 

\hfill$\Box$

\subsection{The case of an $m$-connected domain, $m>3$}
\label{subse:mconnected}
In this subsection we provide a proof of Theorem~\ref{Th:ladder}.
The proof is a modification of Theorem~\ref{Th:pair} with some bookkeeping and 
successive changes of the initial D-BVP.

\noindent{\bf Proof of Theorem 0.4.}
Let $L_1$ be the unique level curve whose existence is guaranteed by Proposition~\ref{pr:onethatenclose}. Let $\Omega_1$ be the domain bounded by $E_1$ and $L_1$. Observe that $\Omega_1$ is homeomorphic to $S^1\times(0,1)$. Suppose that $L_1= l_1\cup l_2\ldots\cup l_k$, where each $l_i$ is a simple cycle (for $i=1,\ldots,k$), where any two of these cycles are either disjoint or intersect at a singular vertex in ${\mathcal T}^{(0)}$ (see Theorem~\ref{th:notsimple}). Let ${\mathcal U}=\{u_1,\ldots,u_m\}$ be the set of all singular vertices along $L_1$ with $\{\mbox{\rm Ind}_g(u_1), \ldots ,  \mbox{\rm Ind}_g(u_m)\}$ being their indices, respectively. 

The interior of each $l_i$, which will be denoted by $l_i^{\circ}$, is a $q_i$-connected domain with $q_i<m-1$, for $i=1,\ldots, k$, (unless $k=1$). We now modify the initial D-BVP on $\Omega_1$, $l_i$ and $l_i^{\circ}$ as described in the discussion preceding Definition~\ref{def:modifiedsolution}, and we obtain $k+1$ new harmonic functions: $g_0$ which is defined on $E_1,\Omega_1$ and $L_1$, and  $g_1,\ldots, g_k$ defined on $l_1,\ldots,l_k$ and their interiors, respectively. 

Recall that modifications are made by adding vertices to  ${\mathcal T}^{(0)}$ along the $l_i$'s and by defining conductance along new edges: those which cross the $l_i$'s (See Equation~(\ref{eq:prehar})). 
By part $(5)$ of Theorem~\ref{th:lengthsareequal} and Remark~\ref{re:lengthof}, we have that the length of $L_1$, with respect to the flux-gradient metric induced by $g_0$, is equal to the length of $E_1$ with respect to the flux-gradient metric induced by $g_0$.  Hence both are equal to the length of $E_1$ with respect to the initial $g$ metric. We record this as
\begin{equation}
\label{eq:lengthis equal}
\mbox{\rm length}_{ g,l_1} (E_1)=\mbox{\rm length}_{ g,l_1} (L_1).
\end{equation}
By using $g_i$ in the interior of each $l_i$,  we now choose a singular curve enclosing all of its $q_i-1$ boundary curves (see Proposition~\ref{pr:onethatenclose}). The result is a set of singular level curves  ${\mathcal W}=\{W_1,\ldots, W_n\}$ and the set of singular vertices ${\mathcal V}=\{v_1,\dots,v_p\}$ they contain with their associated indices ${\mathcal I}=\{\mbox{\rm Ind}_g(v_1), \ldots ,  \mbox{\rm Ind}_g(u_p)\}$. Each simple cycle of a singular level curve $W_j$ contains $s_j<q_i$ boundary curves unless $W_j^{\circ}$ is an annulus. As above, we have only added vertices on the $l_i$'s, assigning each one of the vertices on a specific $l_i$, the  $g(l_i)$-value. The conductance constants are changed only along new edges, those which cross $l_i$, according to  Equation~(\ref{eq:prehar}). Hence, the set ${\mathcal W}$ coincides with the set of singular level curves  of $g$ minus $L_1$  (similar statements hold for ${\mathcal V}$ and ${\mathcal I}$).

We repeat this process (at most) finitely many times, modifying (if needed) the cellular decomposition and defining conductance constants according to Equation~(\ref{eq:prehar}), until the interior of each simple cycle of each singular level curve is an annulus. At each step, we obtain new harmonic functions defined on domains with fewer boundary components.  An equation analogous to Equation~(\ref{eq:lengthis equal}) holds for each simple cycle which is a component of a singular level curve and the nearest singular level curve it contains. That is, its length measured by the flux-gradient metric induced by the harmonic functions defined on its interior is equal to the length of the singular level curve, both measured  by the harmonic map defined in the interior of the simple cycle (see part $(5)$ of Theorem~\ref{th:lengthsareequal}). In turn, they  are equal to the length of the simple cycle measured by the harmonic function defined on its exterior (see Theorem~\ref{le:length of pl}).

It is evident from the proof of Proposition~\ref{pr:onethatenclose}, the disjointness  of level curves which correspond to different $g$-values, and the maximum principle that there is a well-defined hierarchy on the set of singular level curves. Each component of a singular level curve, say $g^{-1}(k_m),\  k_m\in{\mathcal K}$, is contained in the interior of a domain whose one component is a unique simple cycle. This cycle belongs to a unique singular curve, say $g^{-1}(k_n), k_n\in {\mathcal K}, k_n>k_m$ (or $E_1$).  Simple cycles corresponding to level curves are either contained in a domain bounded by a simple cycle which is part of a singular level curve, or are parts of singular level curves as is detailed in  Theorem~\ref{th:notsimple}.

We now turn to the construction of $S_{\Omega}$ and the mapping $f$. We start with a straight Euclidean cylinder, $C_{E_1,L_1}$,  of height $k- g(L_1)$ and circumference which is equal to 
\begin{equation}
\label{eq:lengthoftop}
\sum_{x\in E_1}\frac{\bord g}{\bord n}(x).
\end{equation} 
 Since $L_1$ is a generalized bouquet of circles, we can select a finite number of points in the bottom and identify subsets of them in such a way that the quotient is topologically and metrically isomorphic to $L_1$. 

\begin{figure}[h]
    \begin{center}  
      \includegraphics[width=4in]{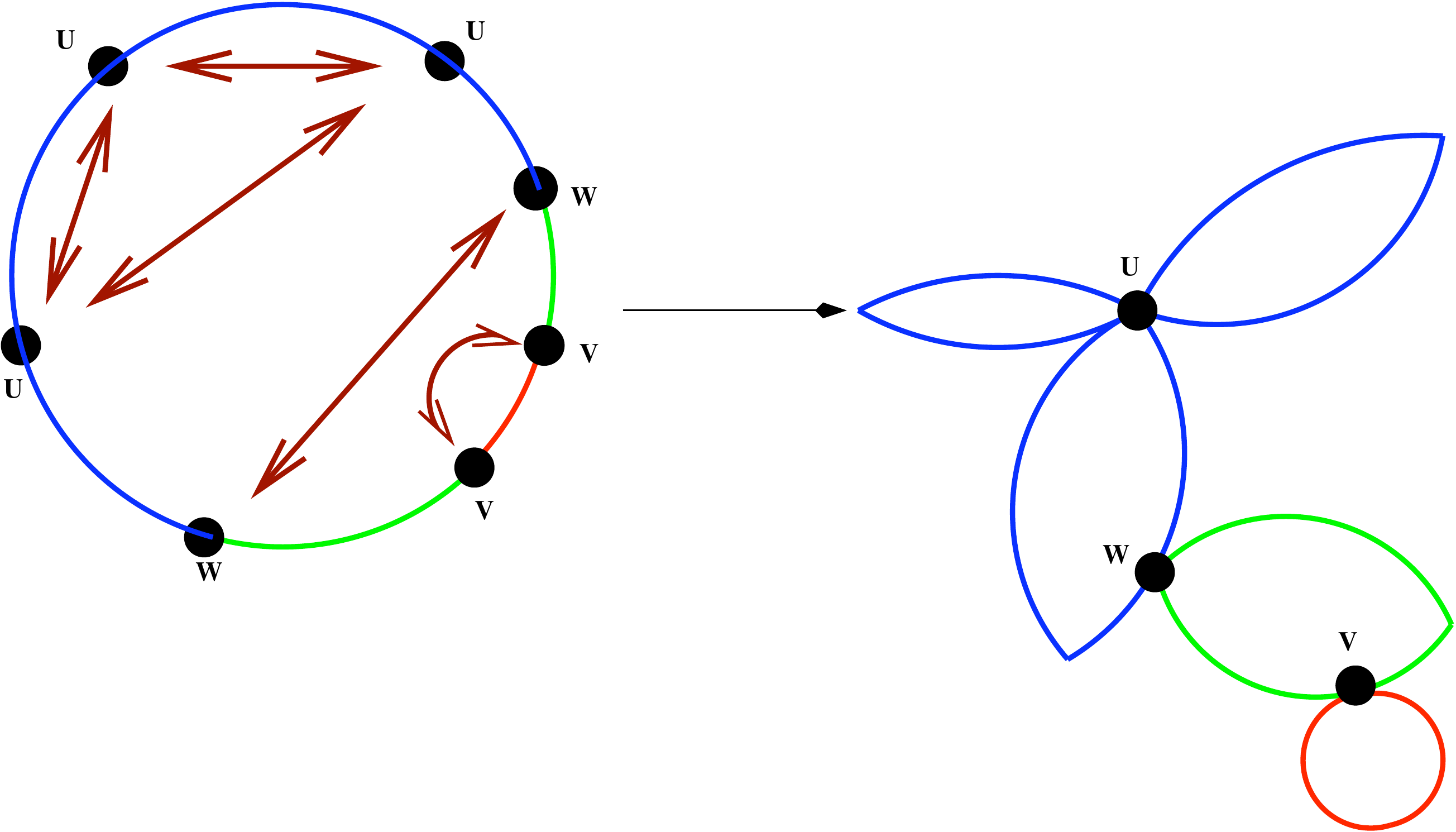}
      \caption{A quotient of a boundary component of a Euclidean cylinder yielding a generalized bouquet of circles}
      \label{fig-Bouqet}
    \end{center}
\end{figure}

As in the proof of Theorem~\ref{Th:pair}, we wish to apply Theorem~\ref{Th:annulus} to $\Omega_1\cup \partial\Omega_1$ ($\partial \Omega_1=E_1\cup L_1$). Since $L_1$ is singular, we may not apply it directly. Instead, we follow the argument carried out in the proof of Theorem~\ref{Th:pair}.  All level curves (which are simple) corresponding to $g(L_1)+\epsilon$, with $\epsilon>0$ being small, will have the same length as $L_1$ (see part $(5)$ of Theorem~\ref{th:lengthsareequal}); hence, as $\epsilon\rightarrow 0$, the sequence of Euclidean cylinders (guaranteed by Theorem~\ref{Th:annulus} applied to the sequence of annuli bounded by these level curves and $E_1$) with their tops being fixed, heights equal to $k- (g(L_1)+\epsilon)$, and circumferences equal to $\sum_{x\in E_1}\frac{\bord g}{\bord n}(x)$, converges to $\tilde C_{E_1,L_1}$ with the top persisting in the limit and the bottom replaced by its quotient as explained above. We denote the mapping given by Theorem~\ref{Th:annulus} and modified as above on the bottom by $f_{\Omega_1}$.

For each $l_i$, $i=1,\ldots k$, there are two cases to consider. First, suppose that $l_i^{\circ}$ does not contain any singular level curve. Without loss of generality, let the unique boundary curve it contains be  $E_2^i$. By assertion $(4)$ of Theorem~\ref{th:lengthsareequal}, we conclude  that the lengths of $E_2^i$ and $l_i$ are equal (measured with the $g_i$-induced flux-gradient metric). Thus, ${\mathcal A}_i=l_i^{\circ}\cup l_i\cup E_2^i$ is an annulus with two (non-singular) boundary components. 
We may therefore apply Theorem~\ref{Th:annulus} and obtain a mapping 
$f_i: {\mathcal A}_i\rightarrow C_{l_i,E_2^i}\  (=S_{{\mathcal A}_{i}})$.
 We now attach the top of the resulting Euclidean cylinder $C_{l_i,E_2^i}$ to $l_i$ by an isometry which is consistent. We conclude that the gluing may be taken as such,  by first arguing that the length of $l_i$ measured with respect to the flux-gradient metric induced by $g_0$,  is the same as its length with respect to the length induced by the flux-gradient metric of $g_i$. This is justified directly by employing Theorem~\ref{le:length of pl} or by using a series of equalities similar to those in Equations~(\ref{eq:length of LI})-(\ref{eq:length of L}) and the reasoning leading to them. The fact that the gluing is consistent follows from the same arguments we have used before.  

Second, assume that $l_i^{\circ}$ contains a singular curve, say $W_i\in{\mathcal W}$. Recall that $W_i$ is the unique singular curve in $l_i^{\circ}$ which encloses all of the boundary curves in  $l_i^{\circ}$. With $l_i$ replacing $E_1$, $W_i$ replacing $L_1$,  and $g_i$ replacing $g_0$, we repeat the construction, yielding a Euclidean cylinder with its bottom being a singular curve as described above (in the case of $L_1$), and a mapping which we will denote by $f_{l_i,{\mathcal W}_i}$. The cylinder  $\tilde C_{l_i,W_i}$ is attached to $\tilde C_{E_1,L_1}$ by gluing its top to the simple cycle corresponding to $l_i$ in the quotient of the bottom of $\tilde C_{E_1,L_1}$, that is, the singular component of $\tilde C_{E_1,L_1}$.

The above analysis is now carried out, at most finitely many times, for each $l_i$, $i=1,\ldots, k$, until we are left with Euclidean cylinders having both of their boundary components non-singular. In particular, the bottom of each cylinder is the image (under the map given by Theorem~\ref{Th:annulus}) of $E_2^{i}$ and its length is equal to
\begin{equation}
\label{eq:lengthofbou}
- \sum_{x\in E_2^i}\frac{\bord g}{\bord n}(x),\  \mbox{\rm for}\  i=1,\cdots, m-1.
\end{equation}
By construction, the images of $E_1$ and $E_2^i$, $i=1,\ldots, m$, under the maps whose construction is described above, are the only boundary components of $S_{\Omega}$. The map $f$ is the obvious union of the collection of mappings constructed above and is analogous to the one appearing in Equation~(\ref{eq:mapping}). It is also evident that $f$ is energy preserving in the sense described before.

We now compute the cone angle, $\phi(v)$, at a singular vertex $v$ with $\mbox{\rm Ind}_g(v)=-n, n\in \NN$. 
By construction, $v$ is the (unique) tangency point of $n+1$ Euclidean cylinders, where $v$ belongs to the non-singular component of each such cylinder. Hence, such a cylinder contributes $\pi$ to the cone angle at $v$. Recall that each such cylinder is attached (by an isometry) to the quotient of the bottom of another Euclidean cylinder. It is easy to check that the contribution to the cone angle at $v$ from this Euclidean cylinder (with one singular boundary component) is equal to $n\times \pi +\pi$.  Therefore, 
\begin{equation}
\label{eq:cone}
\phi(v)=2(n+1)\pi.
\end{equation}
$\hfill\Box$

\bigskip

\begin{Rem}
\label{re:inductionwilldo}
One may also prove Theorem~\ref{Th:ladder} by an induction on the number of boundary components. However, the assertion of Proposition~\ref{pr:onethatenclose} must be used as well as an extension of $g_0$ to $g_1$ over the singular curve $L_1$. Theorem~\ref{le:length of pl} needs to be used in order to prove equality of the $l_1$-length of $L_1$ according to both the $g_0$ and the $g_1$ metric. Overall, we found the proof which does not use induction conceptually more gratifying. We could also stop the process once a planar pair of pants is encountered, thus using Theorem~\ref{Th:pair} directly. Finally, Theorem~\ref{Th:pair} is of course a special case of the theorem above. Still, we maintain that this special case deserves its own proof. 
\end{Rem}

\begin{Rem}
\label{re:valuesofg}
There is a technical difficulty in our construction if some pair of adjacent vertices of ${\mathcal T}^{(0)}$ has the same $g$-value (the first occurrence is in Equation~(\ref{eq:index})).  One may generalize the definitions and the index formula to allow rectangles of area zero, as one solution. 
For a discussion of this approach and others see \cite[Section 5]{Ke}. Experimental evidence  shows that when the cell decomposition is complicated enough, even when the conductance function is identically equal to $1$ and the cells are triangles, such equality rarely happens.
\end{Rem}

\begin{Rem}
\label{re:embedd}
The existence of singular curves for $g$ results in the fact that some rectangles are not embedded in the target. This is evident by the proofs of Theorem~\ref{Th:pair} and Theorem~\ref{Th:ladder}. Since some of the cylinders constructed have a singular boundary component, it is clear that some points in different rectangles that lie on this level curve will map to the same point. However, this occurs only in the situation above and since this fact is not of essential interest to us, we will not go into more details. 
\end{Rem}

 \begin{figure}[h]
 \begin{minipage}[t]{.65\textwidth}
   \begin{center}  
   \includegraphics[width=3.5in]{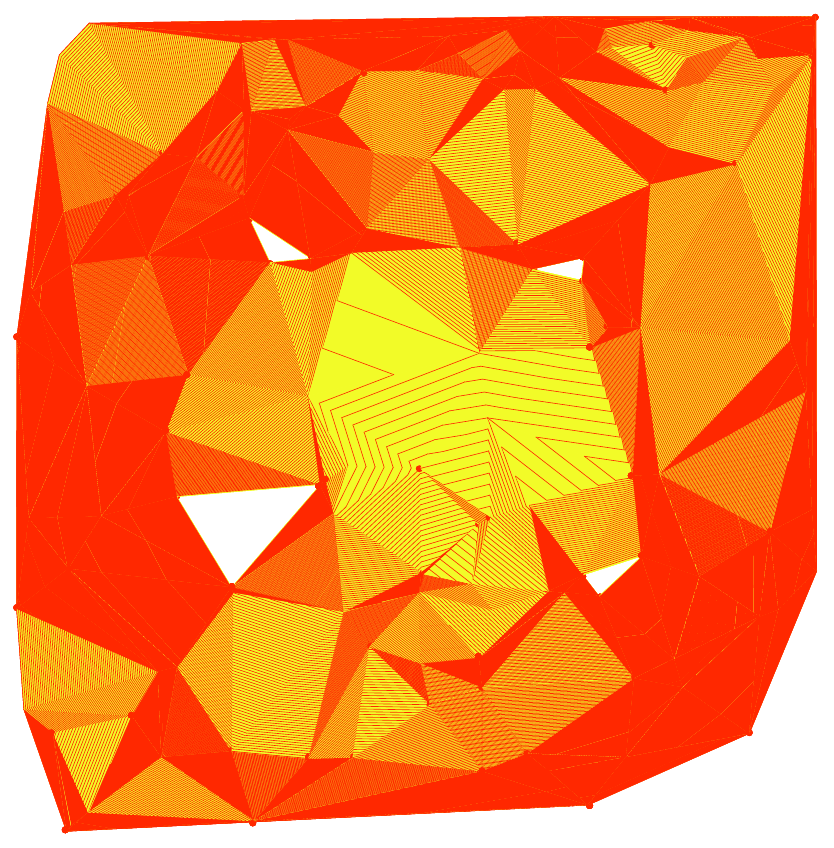}
     \caption{Level curves for a 5 connected domain}
      \label{fig-5levels}
    \end{center}
    \end{minipage}
 \end{figure}

\end{document}